\documentclass[a4paper,11pt]{amsart}
\usepackage[margin=2.8cm]{geometry}

\usepackage[foot]{amsaddr}

\usepackage{hyperref}
\hypersetup{
    colorlinks,
    citecolor=green,
    filecolor=black,
    linkcolor=blue,
    urlcolor=blue
}
\hypersetup{linktocpage}

\usepackage{cite}
\usepackage{graphicx}

\usepackage{amsbsy}
\usepackage{latexsym}
\usepackage{amsfonts}
\usepackage{amssymb}
\usepackage{upgreek}
\usepackage[usenames]{color}
\usepackage{amsmath,amsthm}
\usepackage{enumerate}

\usepackage{tikz}
\usetikzlibrary{arrows}
\usepackage{mathdots}
\usepackage{mathtools}

\usepackage{stmaryrd}
\usepackage{dsfont}

\DeclareMathOperator{\linspan}{span}
\DeclareMathOperator*{\einf}{ess\,inf}

\providecommand{\abs}[1]{\lvert#1\rvert}
\providecommand{\bigabs}[1]{\bigl\lvert#1\bigr\rvert}

\providecommand{\norm}[1]{\lVert#1\rVert}
\providecommand{\bignorm}[1]{\bigl\lVert#1\bigr\rVert}
\providecommand{\Bignorm}[1]{\Bigl\lVert#1\Bigr\rVert}
\providecommand{\biggnorm}[1]{\biggl\lVert#1\biggr\rVert}

\providecommand{\ceil}[1]{\lceil#1\rceil}

\newtheorem{theorem}{Theorem}
\newtheorem{lemma}[theorem]{Lemma}
\newtheorem{prop}[theorem]{Proposition}
\newtheorem{cor}[theorem]{Corollary}
\newtheorem{ass}{Assumptions}

\theoremstyle{definition}
\newtheorem{definition}[theorem]{Definition}
\newtheorem{example}[theorem]{Example}

\theoremstyle{remark}
\newtheorem{remark}[theorem]{Remark}

\newcommand{\cA}{{\mathcal{A}}}

\newcommand{\cM}{{\mathcal{M}}}
\newcommand{\cN}{{\mathcal{N}}}
\newcommand{\cF}{{\mathcal{F}}}

\newcommand{\cV}{\mathcal{V}}
\newcommand{\cT}{\mathcal{T}}
\newcommand{\cI}{\mathcal{I}}
\newcommand{\cS}{\mathcal{S}}
\newcommand{\cO}{\mathcal{O}}

\newcommand{\bA}{\mathbf{A}}
\newcommand{\bC}{\mathbf{C}}
\newcommand{\bM}{\mathbf{M}}

\newcommand{\bu}{\mathbf{u}}
\newcommand{\bv}{\mathbf{v}}
\newcommand{\bw}{\mathbf{w}}
\newcommand{\bbf}{\mathbf{f}}
\newcommand{\bB}{\mathbf{B}}
\newcommand{\br}{\mathbf{r}}
\newcommand{\bd}{\mathbf{d}}
\newcommand{\bT}{\mathbf{T}}
\newcommand{\bs}{\mathbf{s}}

\newcommand{\bP}[1]{\mathbf{P}_{#1}}

\newcommand{\dd}{\mathrm{d}} 
\newcommand{\sdd}{\,\mathrm{d}}

\newcommand{\PP}{\mathbb{P}}

\newcommand{\Chi}{\raise .3ex
\hbox{\large $\chi$}}

\newcommand{\R}{\mathbb{R}}

\newcommand{\N}{\mathbb{N}}
\newcommand{\Z}{\mathbb{Z}}

\newcommand{\meas}{\operatorname{meas}}

\newcommand\dist{\mathop{\rm dist}}
\newcommand\diam{\mathop{\rm diam}}
\newcommand\supp{\mathop{\rm supp}}

\newcommand{\suppF}{\supp_\cF}

\newcommand{\treesupp}{\overline{\supp}^{\mathrm{T}}}

\newcommand{\well}[1]{\ell^\mathrm{w}_{#1}}
\newcommand{\wellp}{\ell^\mathrm{w}_{p}}

\usepackage[section]{algorithm}
\usepackage{algorithmicx}
\usepackage{algpseudocode}
\algrenewcommand\algorithmicrequire{\makebox[46pt][l]{\textrm{required:}}}
\algrenewcommand\algorithmicensure{\makebox[46pt][l]{\textrm{output:}}}
\algrenewcommand\algorithmicfunction{\textrm{function}}
\algrenewcommand\algorithmicwhile{\textrm{while}}
\algrenewcommand\algorithmicdo{}
\algrenewcommand\algorithmicend{\textrm{end}}
\algrenewcommand\algorithmicforall{\textrm{for all}}
\algrenewcommand\algorithmicfor{\textrm{for}}
\algrenewcommand\algorithmicrepeat{\textrm{repeat}}
\algrenewcommand\algorithmicuntil{\textrm{until}}
\algrenewcommand\algorithmicif{\textrm{if}}
\algrenewcommand\algorithmicthen{\textrm{then}}
\algrenewcommand\algorithmicelse{\textrm{else}}

\newcommand{\be}{\begin{equation}}
\newcommand{\ee}{\end{equation}}
\newcommand{\beq}{\begin{eqnarray}}
\newcommand{\beqq}{\begin{eqnarray*}}
\newcommand{\eeq}{\end{eqnarray}}
\newcommand{\eeqq}{\end{eqnarray*}}

\numberwithin{equation}{section}
\numberwithin{theorem}{section}

\theoremstyle{plain}

\title[An adaptive stochastic Galerkin method based on multilevel expansions]{An adaptive stochastic Galerkin method based on multilevel expansions of random fields: Convergence and optimality}
\author{Markus Bachmayr} 
\address{\rm Institut f\"ur Mathematik, Johannes Gutenberg-Universit\"at Mainz, Staudingerweg 9, 55128 Mainz, Germany}
\email[Markus Bachmayr]{bachmayr@uni-mainz.de}
\author{Igor Voulis}
\email[Igor Voulis]{ivoulis@uni-mainz.de}
\thanks{Funded in part by
 Deutsche Forschungsgemeinschaft (DFG, German Research Foundation) -- Projektnummer 233630050 -- TRR 146.}

\date{\today}

\newcommand{\APPLY}{\text{\textsc{Apply}}}
\newcommand{\RESAPPROX}{\text{\textsc{ResApprox}}}
\newcommand{\GALSOLVE}{\text{\textsc{GalSolve}}}
\newcommand{\GALAPPLY}{\text{\textsc{GalApply}}}
\newcommand{\TREEAPPROX}{\text{\textsc{TreeApprox}}}

\newcommand{\M}[1]{#1}

\begin{document}

\maketitle

\begin{abstract}
 The subject of this work is a new stochastic Galerkin method for second-order elliptic partial differential equations with random diffusion coefficients. It combines operator compression in the stochastic variables with tree-based spline wavelet approximation in the spatial variables. Relying on a multilevel expansion of the given random diffusion coefficient, the method is shown to achieve optimal computational complexity up to a logarithmic factor. In contrast to existing results, this holds in particular when the achievable convergence rate is limited by the regularity of the random field, rather than by the spatial approximation order. The convergence and complexity estimates are illustrated by numerical experiments.
 
 \smallskip
 \noindent \emph{Keywords.} parameter-dependent elliptic partial differential equations, stochastic Galerkin method, a posteriori error estimation, adaptive methods, complexity analysis
\smallskip

\noindent \emph{Mathematics Subject Classification.} {35J25, 35R60, 41A10, 41A25, 41A63, 42C10, 65D99, 65N50, 65T60}

\end{abstract}

\section{Introduction}
\label{sec:Intro}

In partial differential equations, one is frequently interested in efficient approximations of the mapping from coefficients in the equations to the corresponding approximate solutions. On a domain $D \subset \R^d$, we consider the elliptic model problem 
\begin{equation}\label{eq:diffusionequation}
   - \nabla \cdot ( a \nabla u ) = f \quad \text{on $D$,} \qquad u = 0 \quad \text{on $\partial D$,}
\end{equation}
where $f \in L_2(D)$ is given, and where we are interested in the dependence of the solutions $u$ on the diffusion coefficients $a$. Especially in the context of uncertainty quantification problems, one considers coefficients $a$ given as random fields on $D$ that can be parameterized by sequences $y = (y_\mu)_{\mu \in \cM}$ of independent scalar random variables $y_\mu$, where typically $\cM = \N$. This leads to the problem of approximating the solutions $u(y)$ for each realization $a(y)$ as a function of the countably many parameters $y$.

A variety of parameterizations of $a$ in terms of random function series have been considered in the literature. 
One instance that has found frequent use in applications are lognormal coefficients $a(y) = \exp( \sum_{\mu\in\cM} y_\mu \theta_\mu)$, where $\theta_\mu$ are functions on $D$ and $y_\mu \sim \cN(0,1)$ are independent. The functions $\theta_\mu$ are typically obtained from a Karhunen-Lo\`eve expansion of a given Gaussian random field. A model case with similar features, on which we focus here, are \emph{affinely parameterized} coefficients:
Assuming $\cM_0$ to be a countable index set with $0 \in \cM_0$ and taking $\cM = \cM_0 \setminus \{ 0 \}$, these are of the form 
\begin{equation} \label{eq:affinecoeff}
a(y) = \theta_0 + \sum_{\mu\in \cM} y_\mu \theta_\mu
\end{equation}
 with $\theta_\mu \in L_\infty(D)$ for $\mu \in \cM_0$, where $\einf_D \theta_0 > 0$. Up to rescaling $\theta_\mu$, we can assume $y_\mu \in [-1,1]$ for each $\mu \in \cM$. The weak formulation of \eqref{eq:diffusionequation} with coefficients \eqref{eq:affinecoeff} then reads: find $u(y) \in V:=H^1_0(D)$ such that
\begin{equation}\label{affinepde}
 	\int_D a(y) \nabla u(y) \cdot \nabla v\sdd x = f(v) \quad \text{\M{for all $v\in V$ and all} $y \in Y := [-1,1]^\cM$,}
\end{equation}
with given $f \in V'$. Well-posedness of the problem for all $y \in Y$ is ensured by the \emph{uniform ellipticity condition}
\begin{equation}\label{uniformellipticity} 
   \einf_D \biggl \{  \theta_0 - \sum_{\mu \in \cM} \abs{\theta_\mu} \biggr\} =: r > 0.
\end{equation}

The subject of this work are numerical methods \M{for computing} approximations of $u(y)$ by sparse product polynomial expansions in the stochastic variables $y$ for given coefficients of the type \eqref{eq:affinecoeff}. 
Methods of this type have been studied quite intensely in recent years; see, for instance, the review articles \cite{Schwab:11,CD} and the references given there.
A central point is that convergence rates can be achieved that depend on the spatial dimension $d$, but not on any dimensionality parameter concerning the parameters $y$. 
The approach of stochastic Galerkin discretizations, which we follow here, is particularly suitable for the construction of adaptive schemes.
Using multilevel structure in the expansion \eqref{eq:affinecoeff}, we obtain a method that converges at rates that are optimal for fully adaptive spatial and stochastic approximations. This holds even for random fields $a$ of low smoothness, with computational costs that scale linearly up to a logarithmic factor with respect to the number of degrees of freedom.

\subsection{Sparse polynomial approximations and stochastic Galerkin methods}
 For simplicity, we assume each $y_\mu$ to be uniformly distributed in $[-1,1]$; different distributions with finite support can be treated with minor modifications. With \M{$\sigma$} the uniform measure on $Y$, we thus consider the mapping $y\mapsto u(y)$ as an element of 
\[
  \cV := L_2(Y, V, \sigma) \simeq V \otimes L_2(Y,\sigma).
\]
\M{With \eqref{uniformellipticity}, it is easy to see that the parameter-dependent solution $u$ of \eqref{affinepde} satisfies $u \in \cV$ and can be equivalently characterized by the variational formulation}
\begin{equation}\label{variationalform}
  \M{\int_Y \int_D a(y) \nabla u(y) \cdot \nabla v(y)\sdd x \sdd\sigma(y) = \int_Y f\bigl(v(y)\bigr)\sdd\sigma(y) \quad \text{for all $v \in \cV$.}}
\end{equation}

From the univariate Legendre polynomials $\{ L_k \}_{k\in\N}$ that are orthonormal with respect to the uniform measure on $[-1,1]$, we obtain (see, e.g., \cite[\S2.2]{Schwab:11}) the orthonormal basis $\{ L_\nu \}_{\nu\in \cF}$ of product Legendre polynomials for $L_2(Y,\sigma)$, which for $y\in Y$ are given by
\[
L_\nu(y) = \prod_{\mu \in \cM} L_{\nu_\mu}(y_\mu) ,\quad \nu\in\cF = \{ \nu \in \N_0^\cM \colon \text{$\nu_\mu \neq  0$ for finitely many $\mu \in \cM$}\} .  
\]
For $u \in \cV$ as in \eqref{variationalform}, we have the basis expansion
\[
   u(y) = \sum_{\nu \in \cF} u_\nu L_\nu(y), \quad u_\nu = \int_Y u(y)\,L_\nu(y)\sdd\sigma(y) \in V .
\]
Restricting the summation over $\nu$ to a finite subset $F\subset \cF$ yields the \emph{semidiscrete} best approximations in $\cV$ by elements of $V \otimes \linspan \{ L_\nu \}_{\nu \in F}$.
Computable approximations are obtained by replacing each $u_\nu$ by an approximation from a finite-dimensional subspace $V_\nu \subset V$ (such as a subspace spanned by finite element or wavelet basis functions). In other words, we seek \emph{fully discrete} approximations of $u$ from spaces
\[
  \cV_{N} = \biggl\{ \sum_{\nu \in F} v_\nu L_\nu \colon  v_\nu \in V_\nu, \nu \in F  \biggr\} \subset \cV
\]
of dimension $N = \sum_{\nu \in F} \dim V_\nu$. In the present work, the spaces $V_\nu$ are chosen as spaces of piecewise polynomial functions of the spatial variables on adaptive grids. Note that due to the selection of the subset $F$, the original problem in countably many parametric dimensions is reduced to a finite but approximation-dependent effective dimensionality. 

The method considered here is based on the stochastic Galerkin variational formulation for $u_N \in \cV_N$,
\begin{equation}\label{stochgalerkin}
  \int_Y \int_D a(y) \nabla u_N(y) \cdot \nabla v(y)\sdd x \sdd\sigma(y) = \int_Y \M{f\bigl(v(y)\bigr)}\sdd\sigma(y) \quad \text{\M{for all} $v \in \cV_N$,}
\end{equation}
again with $a(y)$ as in \eqref{eq:affinecoeff}.
As a consequence of \eqref{uniformellipticity}, the bilinear form given by the left hand side of \eqref{stochgalerkin} is elliptic and bounded on $\cV$, and by C\'ea's lemma 
\[
   \norm{ u_N - u }_\cV \leq \frac{2 \norm{\theta_0}_{L_\infty} - r}{r} \min_{v \in \cV_N} \norm{v - u}_\cV \,,
\]
where we have used that \M{$r \leq a(y) \leq 2 \norm{\theta_0}_{L_\infty} - r$ for all $y \in Y$}.

\subsection{Convergence rates}

The first question in the construction of numerical methods is thus to identify $F$ and $(V_\nu)_{\nu \in F}$ such that $\min_{v \in \cV_N} \norm{u - v}_\cV$ is minimal, up to a fixed constant, for each given computational budget $N$. Under suitable assumptions, one can show that there exist $F$ and $(V_\nu)_{\nu \in F}$ such that
\begin{equation}\label{eq:convrates}
 	\min_{v \in \cV_N} \norm{u - v}_\cV  \leq C N^{-s}
\end{equation}
for some $s>0$, and choosing such $\cV_N$ ensures that the stochastic Galerkin solutions $u_N$ converge at the same rate.
One now aims to realize this choice by adaptive methods that only use the problem data $D$, $f$, and the expansion \eqref{eq:affinecoeff} of $a$ as input. These methods should also be \emph{universal}, that is, they should not require knowledge of $s$ in \eqref{eq:convrates}, but rather automatically realize the best possible rate $s$ for each given problem. A basic building block for such methods are computable \emph{a posteriori error estimates} for $u_N$.
Beyond the convergence of the computed approximations at optimal rates with respect to $N$, in practice the computational costs of constructing $\cV_N$ and $u_N$ are crucial. An adaptive method is said to be of \emph{optimal complexity} if the required number of elementary operations (and hence the computational time) is bounded by a fixed multiple of $N$.

As the basic approximability results in \cite{BCM,BCDS} show, the type of expansion \eqref{eq:affinecoeff} of the random field $a(y)$ plays a role in the rate $s$ that is achievable in \eqref{eq:convrates}. In contrast to Karhunen-Lo\`eve-type expansions in terms of functions $\theta_\mu$ with global supports on $D$, improved results can be obtained for expansions with $\theta_\mu$ that have localized supports. In particular, this is the case for $\theta_\mu$ with wavelet-type \emph{multilevel structure}, which we focus on in this work. 
To each $\mu \in \cM$ we assign a \emph{level} $|\mu|=\ell \in \N_0$. 
We assume $\theta_\mu$ to have the properties that there exists $C_1>0$ such that
 \begin{equation}\label{multilevel1}
\#\{ \mu : |\mu| = \ell \} \leq C_1 2^{d\ell} \quad \M{\text{for all $\ell \geq 0$}},
\end{equation}
and there exists $C_2>0$ such that for some $\alpha>0$,
\begin{equation}
\label{multilevel2}
  \sum_{|\mu|=\ell} \abs{\theta_{\mu}} \leq C_2 2^{-\alpha  \ell} \quad \text{\M{a.e.~in $D$, for all $\ell \geq 0$}}.
\end{equation}
Expansions of this type for several important classes of Gaussian random fields are constructed in \cite{Gittelson:12,BCM:18}, and it is thus natural to use such these also in the model case of affine parameterizations. For sufficiently regular $\theta_\mu$, the parameter $\alpha$ can be seen to correspond to the H\"older regularity of realizations of the random field $a(y)$.
Note that for multilevel basis functions, the condition \eqref{uniformellipticity} is less restrictive than for globally supported $\theta_\mu$; in particular, in the multilevel case, any H\"older smoothness index $\alpha>0$ is possible in \eqref{multilevel2}.

However, $s$ in \eqref{eq:convrates} is also constrained by the spatial regularity of the further problem data $f$ and $D$, as well as by the permissible choices of spaces $V_\nu$. The simplest option is to choose all $V_\nu$ equal to the same sufficiently rich subspace of $V$.
Several approximation results and adaptive schemes in the literature are based on choosing each $V_\nu$ from a fixed hierarchy of nested subspaces of $V$, such as wavelet subspaces or finite element spaces corresponding to uniformly refined meshes (see, e.g., \cite{CDS:11,Gittelson:13,MR3952679}).
For multilevel expansions with properties \eqref{multilevel1}, \eqref{multilevel2}, the results in \cite[\S8]{BCDS} show a potential advantage of choosing $V_\nu$ adapted specifically for each $\nu$, for instance by a separate adaptive finite element mesh for each $\nu$. 
For $d\geq 2$ and $\alpha \in (0,1]$, these results yield a rate $s = \frac{\alpha}{d} - \delta$ for any $\delta > 0$ in \eqref{eq:convrates}. 
Remarkably, this rate for fully discrete approximation is the same as established in \cite{BCM} for only semidiscrete approximation.
As noted in \cite{BCD}, this is also the same rate as for spatial approximation of a single realization of $u(y)$ in $H^1$ for $y \in Y$ drawn uniformly at random.
In other words, in this setting, the full stochastic dependence can be approximated at the same rate as a single realization of the random solution. This is related to the multilevel structure of the $\theta_\mu$ also reappearing to a certain degree in the coefficients $u_\nu$, but in a strongly $\nu$-dependent way that necessitates individually adapted spaces $V_\nu$.
 
\subsection{New contributions and relation to previous results}
In this work, we prove a new adaptive stochastic Galerkin scheme to have optimal computational complexity, up to a logarithmic factor, in realizing this convergence rate. To the best of our knowledge, this is the first such result for the case where the approximability is limited by \M{the decay in absolute value of the functions $\theta_\mu$ in the random field expansion (that is, by the smoothness parameter $\alpha$ in \eqref{multilevel2})} rather than by the approximation order of the spatial basis functions. In particular, we improve on a previous result based on wavelet operator compression from \cite{BCD}: the method analyzed there yields suboptimal rates that get closer to $\alpha/d$ for more regular spatial wavelet basis functions. For practically realizable degrees of regularity of the basis, however, the resulting rates for this previous method remain rather far from optimal.

Note that the situation is different when $\alpha$ is large in comparison to the approximation order of the spatial basis functions. In this case, which corresponds to a more rapidly convergent expansion \eqref{eq:affinecoeff}, the rate $s$ in \eqref{eq:convrates} is constrained, independently of $\alpha$, by the spatial approximation rate. In such a setting, optimality with respect to this spatial rate is obtained by the adaptive scheme from \cite{MR3164144}, which is also based on wavelet operator compression. In the present work, however, we focus on the case of sufficiently high-order spatial approximation such that the achievable rate $s$ is determined by the random field $a(y)$.

Many existing methods use spatial approximations by finite elements, for instance, as in \cite{MR3154028,MR3177362,MR3423228,MR3519560,MR4114136,MR3952679,BPR21}.
Convergence and complexity of such methods, however, has been established only to a more limited extent than for wavelet approximations. 
For a method using a single adaptively refined finite element mesh, convergence and quasi-optimal cardinality of this spatial mesh are shown in \cite{MR3423228}.
In contrast, independently adapted meshes are used in \cite{MR3154028} and \cite{MR3952679}. In the latter case, meshes for each Legendre coefficient are selected from a fixed refinement hierarchy. 
The method in \cite{MR3952679} as well as the analysis in \cite{MR4014787} rely on an unverified saturation assumption.
In \cite{BPR21}, a method using a separately adapted mesh for each Legendre coefficient is shown to produce approximations converging at optimal rates. However, this is done using a further strengthened saturation assumption, and there are no bounds on the computational complexity.
These finite element-based methods are all constructed for $\theta_\mu$ of general supports and do not make use of multilevel expansions of random fields. 

The main component of our new method is a scheme for error estimation by sufficiently accurate approximation of the full spatial-stochastic residual. For achieving improved computational complexity, it makes crucial use of the multilevel structure \eqref{multilevel1}, \eqref{multilevel2}. 
The spatial discretization is done by spline wavelets.
We combine a semidiscrete adaptive operator compression on the stochastic degrees of freedom, which is independent of the spatial discretization, with a tree-based evaluation of spatial residuals. 
In the latter step, we use that the spatial coefficients are approximated by piecewise polynomials, evaluate the wavelet coefficients using a multi-to-single-scale transform following \cite{S14}, and use tree coarsening (based on a modification of a result in \cite{B:07,B:18}) to identify new degrees of freedom by a bulk chasing criterion. With these ingredients at hand, the adaptive scheme can be constructed similarly to the ones in \cite{GHS} and \cite{S14}.
Due to the use of operations on trees, the complexity estimates for our method rely on tree approximability for the Legendre coefficients $u_\nu$.

The near-optimality result for our method can be summarized as follows: if the best fully discrete approximation $u_N$ with spatial tree structure in each Legendre coefficient requires a total number of $N = \mathcal{O}(\varepsilon^{-1/s})$ degrees of freedom for an error bound $\norm{ u - u_N }_\cV \leq \varepsilon$, then our method finds an approximation satisfying this error bound using $\mathcal{O}(\varepsilon^{-1/s} \abs{\log \varepsilon })$ arithmetic operations.
In addition, we show that for best approximations \M{with spatial tree structure}, one obtains the same convergence rates of best approximations as shown in \cite[\S8]{BCDS} for general sparse approximations. 
Altogether, this shows that for $\alpha\in(0,1]$ and $d\geq 2$, for all $s < \frac{\alpha}{d}$ the method requires $\mathcal{O}(\varepsilon^{-1/s})$ operations; in the special case $d=1$ this holds for all $s < \frac{2}{3} \alpha$. These results are confirmed by our numerical tests, which indicate that these statements continue to hold true for $\alpha > 1$.

The regularity requirements on the problem data are the same as for the underlying approximability statements from \cite{BCDS}, and unlike \cite{BCD}, the wavelet basis functions are only required to be $C^1$ splines.
The use of wavelets in this scheme allows us to avoid a number of technicalities in its analysis that would arise with finite element discretizations. However, in contrast to the existing methods with computational complexity bounds from \cite{BCD,MR3164144}, our basic strategy is generalizable to spatial approximation by finite elements.

\subsection{Outline and notation}

 In Sec.~2, we state our main assumptions on the problem data in \eqref{affinepde} and review the relevant approximability results for solutions. In Sec.~3, we discuss the basic construction of stochastic Galerkin schemes that our new method is based on and recapitulate a related previous operator compression result that leads to a suboptimal method. In Sec.~4, we describe the new residual approximation using tree approximation in the spatial discretization, a corresponding tree coarsening scheme, and solver for Galerkin discretizations. In addition, we verify that the sought solution has the required slightly stronger tree approximability. In Sec.~5, we analyze convergence and computational complexity of the resulting adaptive method. In Sec.~6, we illustrate these results by numerical experiments. We conclude with a summary of our findings and an outlook on further work in Sec.~7. 

By $A \lesssim B$, we denote that there exists $C>0$ independent of the quantities appearing in $A$ and $B$ such that $A \leq CB$. Moreover, we write $A \gtrsim B$ for $B\lesssim A$ and $A \sim B$ for $A\lesssim B\wedge B\lesssim A$. By $\meas(S)$, we denote the Lebesgue measure of a a subset $S$ of Euclidean space.
Where this cannot cause confusion, we write $\norm{\cdot}$ for the $\ell_2$-norm on the respective index set and $\langle \cdot, \cdot\rangle$ for the corresponding inner product.

\section{Sparse Approximations and Stochastic Galerkin Methods}

In this section, we summarize the results on convergence rates of sparse polynomial  approximations from \cite{BCM,BCDS} for coefficient expansions \eqref{eq:affinecoeff} in terms of functions $\theta_\mu$, $\mu\in\cM$, with multilevel structure.
While $\abs{\mu}$ describes the scale of $\theta_\mu$, for each fixed $|\mu|$, the index $\mu$ determines the spatial localization of this function.
Conditions \eqref{multilevel1} and \eqref{multilevel2} are satisfied in particular when $\theta_\mu$ correspond to a rescaled, level-wise ordered wavelet-like basis with following properties.

\begin{ass}\label{ass:wavelettheta}
We assume $\theta_\mu \in W^1_\infty(D)$ for $\mu \in \cM_0$ such that in addition to \eqref{multilevel1}, the following hold for all $\mu \in \cM$:
\begin{enumerate}[{\rm(i)}] 
\item $\diam \supp \theta_\mu\sim 2^{-|\mu|}$, 
\item there exists $M>0$ such that for each $\mu$,
\[
  \#\{ \mu' \in \cM \colon |\mu|=|\mu'|,\, \supp \theta_{\mu}\cap \supp \theta_{\mu'} \} \leq M ,
\] 
\item for some $\alpha>0$, one has $\norm{\theta_\mu}_{\M{L_\infty(D)}} \lesssim 2^{-\alpha |\mu|}$.
\end{enumerate}
\end{ass}

\subsection{Semidiscrete approximations}
We first consider sparse Legendre approximations of $u(y) \in V$ with respect to the parametric variables $y \in Y$.
For given $n \in \N$, selecting $F_n \subset \cF$ to comprise the indices of $n$ largest $\norm{u_\nu}_V$ yields the best $n$-term approximation of $u$ by product Legendre polynomials,
\[
   u_{F_n} := \sum_{\nu\in F_n} u_\nu \, L_\nu.
\]
The error in $\cV$ of approximating $u$ by $u_F$ decays with rate $\mathcal{O}(n^{-s})$ precisely when \M{the sequence $\bigl( \norm{u_\nu}_V \bigr)_{\nu\in\cF}$ is an element of the linear space $\cA^s(\cF)$ of sequences with finite quasi-norm}
\begin{equation}
\label{stdAsdef1}
  \bignorm{\bigl( \norm{u_\nu}_V \bigr)_{\nu\in\cF}}_{\cA^s(\cF)} := \sup_{n\in\N_0} (n+1)^s
  \inf_{\substack{  F \subset \cF \\   \# F \leq n}} \Bigl( \sum_{\nu \in \cF\setminus F} \norm{u_\nu}_V^2 \Bigr)^{\frac12}  \,.
\end{equation}
As a consequence of the Legendre coefficient estimates in \cite{BCM}, we have the following approximability result, which is an immediate consequence of \cite[Cor.~4.2]{BCM}.

\begin{theorem}
Let \eqref{uniformellipticity} as well as \eqref{multilevel1}, \eqref{multilevel2} hold. Then
\[
  \bigl( \norm{ u_\nu }_V \bigr)_{\nu \in \cF} \in \cA^s(\cF) \quad\text{for any $s < \frac{\alpha}d$.}
\]
\end{theorem}

Inserting product Legendre expansions of \M{$u,v$ into \eqref{variationalform}} leads to the \emph{semidiscrete form} of the stochastic Galerkin problem for the coefficient functions $u_\nu$, $\nu \in\cF$,
\begin{equation}\label{semidiscreteform}
 \sum_{\mu \in \cM_0} \sum_{\nu' \in \cF} (\bM_\mu)_{\nu,\nu'} A_\mu u_{\nu'}  =  \delta_{0, \nu} f ,
  \quad \nu \in \cF,
\end{equation}
where $A_\mu\colon V\to V'$ are defined by
\[
  \langle A_\mu v,w\rangle := \int_D \theta_\mu \nabla v \cdot \nabla w\sdd x \quad v,w\in V,\; \mu \in \cM_0,
\]
and the mappings $\bM_\mu \colon \ell_2(\cF)\to\ell_2(\cF)$ are given by
\[
\begin{aligned}
 \bM_0 & := \left( \int_{Y}  L_\nu(y) L_{\nu'}(y) \sdd\sigma(y)  \right)_{\nu,\nu'\in\cF}, \\
 \bM_\mu & := \left( \int_{Y} y_\mu  L_\nu(y) L_{\nu'}(y) \sdd\sigma(y)  \right)_{\nu,\nu'\in\cF},\quad \mu \in \cM.
 \end{aligned}
\]
Since the $L_2([-1,1],\frac12 \sdd y)$-orthonormal Legendre polynomials $\{ L_k\}_{k\in \N}$ satisfy the three-term recursion relation
\[
    y L_k(y) =\sqrt{\beta_{k+1}} L_{k+1}(y) +  \sqrt{\beta_k} L_{k-1}(y), \quad\beta_k = (4 - k^{-2})^{-1} ,
\]
with $L_0 = 1$, $L_{-1} = 0$, $\beta_0 = 0$, we have
\[
\begin{aligned}
   \bM_0 &= \bigl( \delta_{\nu,\nu'} \bigr)_{\nu,\nu' \in \cF}, \\ 
   \bM_\mu &= \biggl( \sqrt{\beta_{\nu_\mu+1}} \,\delta_{\nu+e_\mu, \nu'} + \sqrt{\beta_{\nu_\mu}}\, \delta_{\nu-e_\mu, \nu'} \biggr)_{\nu,\nu' \in \cF},\; \mu \in \cM,
\end{aligned}
\]
with the Kronecker vectors $e_\mu = ( \delta_{\mu,\mu'} )_{\mu' \in \cM}$.

\subsection{Fully discrete approximations}\label{sec:fullseqspace}
We now turn to additional spatial approximation.
Let $\Psi := \{ \psi_\lambda \}_{\lambda \in \cS}$ with a countable index set $\cS$ be a Riesz basis of $V$,
\begin{equation}\label{rieszV}
  	 c_\Psi \norm{ \bv }_{\ell_2(\cS)} \leq \biggnorm{\sum_{\lambda \in \cS} \bv_\lambda \psi_\lambda }_V \leq C_\Psi \norm{ \bv }_{\ell_2(\cS)}. 
\end{equation}
We can then expand $u$ in terms of its coefficient sequence $\bu \in \ell_2(\cF\times\cS)$ as
\begin{equation}\label{fullydiscreteexpansion}
    u = \sum_{\substack{\nu \in \cF \\ \lambda \in \cS}} \bu_{\nu,\lambda} \, L_\nu \otimes \psi_\lambda,
\end{equation}
where we write $
  \bu_\nu = \bigl( \bu_{\nu,\lambda}   \bigr)_{\lambda\in \cS}\,
$.
Note that by duality, we also have
\begin{equation}
    c_\Psi   \norm{ g }_{V'} \leq   \bignorm{ \bigl(  g(\psi_\lambda)  \bigr)_{\lambda \in \cS} }_{\ell_2} \leq  C_\Psi  \norm{ g }_{V'} \,, \quad g \in V'\,.
\end{equation}
The variational problem \eqref{variationalform} can equivalently be rewritten as an operator equation \M{on the sequence space $\ell_2(\cF\times\cS)$ in the form}
\begin{equation}\label{seqform}
   \bB \bu :=  \sum_{\mu \in \cM_0} ( \bM_\mu \otimes \bA_\mu ) \bu = \bbf,
\end{equation}
where
\begin{equation}\label{eq:Amudef}
\bA_\mu := \bigl( \langle A_\mu  \psi_{\lambda'}, \psi_{\lambda}\rangle \bigr)_{\lambda,\lambda' \in \cS}\,, \quad \mu \in \cM_0,
\qquad\bbf := \big(\langle f, L_\nu  \otimes  \psi_\lambda\rangle\big)_{(\nu,\lambda)\in \cF\times\cS}.
\end{equation}

In what follows, we assume $\Psi$ to be a sufficiently smooth wavelet-type basis of approximation order greater than one. Here each index $\lambda \in \cS$ comprises the level $\abs{\lambda}$ of the corresponding basis element, its position in $D$, and the wavelet type. We assume that $\diam \supp \psi_\lambda \sim 2^{-\abs{\lambda}}$ for $\lambda \in \cS$ and, without loss of generality, $\min_{\lambda \in \cS} \abs{\lambda} = 0$.

In the case of fully discrete approximations based on expansions \eqref{fullydiscreteexpansion} with the spatial Riesz basis $\Psi$, the relevant type of sparsity is quantified by the quasi-norms,
\begin{equation}
\label{stdAsdef2}
  \norm{\bv}_{\cA^s(\cF \times \cS)} := \sup_{N\in\N_0} (N+1)^s
  \inf_{\#\supp \bw\leq N} \norm{\bv - \bw}_{\ell_2(\cF\times \cS)}.
\end{equation}
Note that here, \M{$\supp \bw = \{ (\nu,\lambda) \in \cF\times \cS \colon \bw_{\nu,\lambda} \neq 0 \}$} is chosen from arbitrary subsets of $\cF\times \cS$, so that each Legendre coefficient of the corresponding element of $\cV$ is approximated with an independent adaptive spatial approximation.

For any $s>0$ and a countable index set $\cI$, for $p>0$ given by $p^{-1} = s + \frac12$ the space $\cA^s(\cI)$ can be identified with the weak-$\ell_p$ space $\wellp(\cI)$. The corresponding quasi-norm
\[
  \norm{\bw}_{\wellp} = \sup_{k \in \N} k^{1/p} \bw^*_k,
\]
where $\bw^*_k$ is the $k$-th largest of the numbers $\abs{\bw_\lambda}$, $\lambda \in \cI$, satisfies
\begin{equation}\label{Asequiv}
  \norm{\bw}_{\wellp}  \sim  \norm{\bw}_{\cA^s}
\end{equation}
with constants depending only on $s$. Moreover, note that for all $p, \varepsilon >0$, one has
\begin{equation}\label{wlpnested}
  \ell_p \subset \wellp \subset \ell_{p+\varepsilon}.
\end{equation}

In what follows, we use a basic approximability result established in \cite{BCDS}. Note that the assumptions given here are not the sharpest possible, but allow us to avoid some technicalities. 

\begin{theorem}\label{thm:W2tau}
In addition to \eqref{uniformellipticity} and Assumptions \ref{ass:wavelettheta} with levelwise decay rate $\alpha>0$, let $D$ be convex, $f \in L^2(D)$, and $\norm{\nabla \theta_\mu}_{L_\infty} \lesssim 2^{-(\alpha-1)\abs{\mu}}$ for $\mu \in \cM$.
Let $\alpha \in (0,1]$ and $\tau \in (1,2]$. Then for any $\hat\alpha \in (0,\alpha)$, with $Z_{\hat \alpha} :=  V \cap [ H^1(D), W^2_\tau(D) ]_{\hat\alpha}$, one has
\[
  \sum_{\nu \in \cF} \norm{ u_\nu }_{Z_{\hat\alpha}}^p < \infty
\]
for any $p > 0$ such that
\[
  \frac{1}{p} < \frac{\alpha}{d} + \frac12 
   + \left( \frac{1}{\tau} - \frac12 - \frac{1}{d} \right) \hat\alpha .
\]
\end{theorem}

As a consequence of \cite[Prop.~7.4]{BCDS}, the complex interpolation space $Z_{\hat\alpha}$ has the following approximation property: there exists $C>0$ such that for all $ v \in Z_{\hat\alpha}$,
\begin{equation}\label{Zthetaapprox}
  \inf \bigl\{  \norm{ v - v_n }_V \colon\;  v_n  \in \linspan \{\psi_\lambda\}_{\lambda \in S},\, S \subset \cS, \,\# S \leq n  \bigr\} 
  \leq C n^{-{\hat\alpha} / d} \norm{ v }_{Z_{\hat\alpha}}.
\end{equation}
Note that an analogous property holds when the wavelet approximations are replaced by adaptive finite elements.
With appropriately chosen $\tau$ and $\hat\alpha$, by the arguments in \cite[Section 8.2]{BCDS} this implies in particular the following.

\begin{cor}\label{cor:fullapprox}
Let the assumptions of Theorem \ref{thm:W2tau} hold, and let $d\in \{2,3\}$. Then
\begin{equation}\label{fullapprox1}
   \sum_{\nu \in \cF} \norm{ \bu_{\nu}   }_{\cA^s(\cS)}^p < \infty\quad \text{for any $p,s>0$ such that $\displaystyle \frac1p < \frac{\alpha}d + \frac12$ and $\displaystyle s< \frac{\alpha}d$.}
\end{equation}
\end{cor}

In view of \eqref{Asequiv} and \eqref{wlpnested}, the bound \eqref{fullapprox1} in turn implies
\begin{equation}\label{fullapprox2}
  \sum_{\nu \in \cF} \abs{\bu_{\nu,\lambda}}^p < \infty\quad \text{for any $p>0$ such that $\displaystyle \frac1p < \frac{\alpha}d + \frac12$,}
\end{equation}
and as a further consequence
\begin{equation}\label{eq:fullapproxAs}
    \bu \in \cA^s(\cF\times \cS) , \quad \text{for any $s >0$ such that $\displaystyle s < \frac\alpha{d}$.}
\end{equation}
As a consequence, for this type of \M{fully discrete best $N$-term approximation} we remarkably have the same limiting convergence rate as for the semidiscrete Legendre approximation and for approximating $u(y)$ for a single random draw of $y$. 

\begin{remark}\label{rem:1dcase}
In the special case $d=1$, since the above results do not apply to $\tau <1$, we obtain \eqref{fullapprox2} only with $\frac1p < \frac23 \alpha + \frac12$, corresponding to $s < \frac23\alpha$ for $\alpha \in (0,1]$.
\end{remark}

\section{Adaptive Stochastic Galerkin Methods}

We now review basic concepts of adaptive stochastic Galerkin schemes in terms of the \M{sequence space} formulation \eqref{seqform} as well as \M{the previous} results on an adaptive method with complexity bounds from \cite{BCD}. In what follows, we write $\norm{\cdot}$ for the $\ell_2$-norm on the respective index set and $\langle \cdot, \cdot\rangle$ for the corresponding inner product.

\subsection{Stochastic Galerkin discretization} 
Under the assumption \eqref{uniformellipticity}, for the self-adjoint mapping $\bB$ on $\ell_2(\cF\times \cS)$, with $r_\bB := c_\Psi^2 r$ and $R_\bB := C_\Psi^2(2\norm{\theta_0}_{L_\infty} - r)$, we have
\begin{equation}\label{eq:Bellipt}
  r_\bB \norm{\bv}^2 \leq \langle \bB \bv, \bv \rangle \leq R_\bB \norm{\bv}^2, \qquad \bv\in \ell_2(\cF\times\cS).
\end{equation}
For any $\Lambda \subset \cF \times \cS$, the corresponding stochastic Galerkin approximation is defined as the unique $\bu_\Lambda$ with $\supp \bu_\Lambda \subseteq \Lambda$ such that
\[
   \bigl(\bB \bu_{\Lambda} - \bbf \bigr)|_{\Lambda} = 0.
\]
By \eqref{eq:Bellipt}, this system of linear equations in $\#\Lambda$ unknowns has a symmetric positive definite system matrix with spectral norm condition number bounded, independently of $\Lambda$, by $\kappa(\bB) = \norm{\bB} \norm{\bB^{-1} } \leq R_\bB / r_\bB$. It can thus be solved to the required accuracy, for instance, by direct application of the conjugate gradient method.

In the convergence analysis of adaptive methods based on solving successive Galerkin problems, the following \emph{saturation property} plays a crucial role; for the proof, see \cite[Lmm.~4.1]{CDD01} and \cite[Lmm.~1.2]{GHS}.

\begin{lemma}  \label{lmm:saturation}
  Let $ \omega \in (0,1]$, $\bw \in \ell_2( \cF\times\cS)$, $\Lambda \subset \cF\times \cS$ such that $\supp \bw \subset \Lambda$ and
  \begin{equation}\label{eq:doerfercond}
    \norm{ ( \bB \bw - \bbf)|_{\Lambda}  } \geq \omega \norm{ \bB \bw - \bbf } ,
  \end{equation}
and let $\bu_\Lambda$ with $\supp \bu_\Lambda \subseteq \Lambda$ be the solution of the Galerkin system $(\bB \bu_\Lambda - \bbf)|_\Lambda = 0$. 
Then
\begin{equation}\label{eq:saturation}
  \norm{ \bu - \bu_\Lambda }_\bB \leq \left( 1 - \frac{\omega^2}{\kappa(\bB)}  \right)^{\frac12} \norm{ \bu - \bw }_\bB  ,
\end{equation}
where $\norm{\bv}_\bB = \sqrt{ \langle \bB \bv,\bv\rangle }$ for $\bv \in \ell_2(\cF\times \cS)$.
\end{lemma}

Note that whereas a saturation property of the type \eqref{eq:saturation} is \emph{assumed} in \cite{MR3952679,MR4014787} and in a further strengthened form for the rate estimates in \cite{BPR21}, as a consequence of Lemma \ref{lmm:saturation}, no such assumption is required in the present case.

\subsection{Adaptive Galerkin method}
In its basic idealized form, the adaptive Galerkin scheme that was analyzed in \cite{GHS} in the context of wavelet approximation is performed in two steps. In our setting, for each $k\in\N$, in step $k$ of the scheme we are given $F^k \subset \cF$ and $S^k_\nu \subset \cS$ for $\nu \in F^k$ and find $F^{k+1}$ and $(S^{k+1}_\nu)_{\nu \in F^{k+1}}$ as follows:
\begin{itemize}
  \item Solve the Galerkin problem on $\Lambda^k := \{ (\nu,\lambda) \colon \nu \in F^k, \lambda \in S^k_\nu \}$ to obtain $\bu^k$ with $\supp \bu^k \subseteq \Lambda^k$ satisfying $(\bB \bu^k - \bbf )|_{\Lambda^k} = 0$.
  \item \M{Choose $\Lambda^{k+1}$ as the smallest set $\hat\Lambda \subset \cF \times \cS$ such that $\norm{(\bB \bu^k - \bbf)|_{\hat\Lambda}} \geq \omega \norm{\bB \bu^k - \bbf}$,} where $\omega \in (0,1]$ is fixed and sufficiently small.
\end{itemize} 
This basic strategy is also known as \emph{bulk chasing}; the condition \M{$\norm{(\bB \bu^k - \bbf)|_{\hat\Lambda}} \geq \alpha \norm{\bB \bu^k - \bbf}$} is analogous to \emph{D\"orfler marking} in the context of adaptive finite element methods.
For arriving at a practical scheme, the main difficulty lies in this second step, since the sequences \M{$\bB \bu^k - \bbf$} in general have infinite support. 
One thus needs to replace \M{$\bB \bu^k - \bbf$} by finitely supported approximations. In addition, the required Galerkin solutions are computed only inexactly. The condition of $\Lambda^{k+1}$ being selected to have minimal cardinality can also be relaxed, which is crucial when using approximations with additional tree structure constraints.

The numerically realizable version of the adaptive Galerkin method given in Algorithm \ref{AGM} relies on two problem-dependent procedures. 
The first, invoked in step \textbf{(i)}, consists in a method for computing a finitely supported approximation $\br^k$ of $\bB\bu^k - \bbf$ of sufficient relative accuracy.
The second, used in step \textbf{(iii)}, is a scheme for the approximate solution of Galerkin problems on the index sets \M{that are} determined in a problem-independent manner in step \textbf{(ii)} from $\br^k$ to satisfy a bulk-chasing criterion.

For the latter step, following \cite{S14}, we use a substantially relaxed version of the minimality requirement on $\Lambda^{k+1}$ that is appropriate for tree approximation. In the context of standard sparse approximation as in \cite{CDD01,GHS}, one may take $\omega_0=\omega_1$ and select $\Lambda^{k+1}$ by directly adding the indices corresponding to the largest entries of $\br^k$ to $\Lambda^{k}$.

\begin{algorithm}[t]
	\caption{Adaptive Galerkin method}\label{AGM}
	\flushleft Let $0< \omega_0 \leq \omega_1 < 1$, $\zeta, \gamma >0$, $\bu^0 = 0$, and $\Lambda^0 = \emptyset$. 
	\flushleft For $k = 0, 1, 2, \ldots$, perform the following steps:
\begin{enumerate}[{\bf(i)}]
\item Find $\br^k$ with $\#\supp \br^k < \infty$ such that $\norm{ \br^k - (\bB \bu^k - \bbf) }  \leq \zeta \norm{\bB\bu^k - \bbf}$
\item Find $\Lambda^{k+1}$ satisfying
\begin{subequations}
  \makeatletter
  \def\@currentlabel{\M{A3.1.1}}
  \makeatother
\label{eq:agmbulk}
\begin{align}
   \norm{\br^k|_{\Lambda^{k+1}}} &\geq \omega_0 \norm{\br^k}, \label{bulk1} \tag{\M{A3.1.1a}} \\
  \#(\Lambda^{k+1} \setminus \Lambda^k) &\lesssim \#(\tilde\Lambda\setminus \Lambda^k) \quad\text{for any $\tilde\Lambda \supset \Lambda^k$ 
  such that $\norm{\br^k|_{\tilde\Lambda} }\geq \omega_1 \norm{\br^k}$} \tag{\M{A3.1.1b}}\label{bulk2}
\end{align}
\end{subequations}
\item Find $\bu^{k+1}$ such that $\norm{(\bB \bu^{k+1} - \bbf)|_{\Lambda^{k+1}}} \leq \gamma \norm{\br^k}$ with $\supp \bu^{k+1} \subseteq \Lambda^{k+1}$
\end{enumerate}
\end{algorithm}

\subsection{Previous results on direct fully discrete residual approximations}

A standard construction for the approximate evaluation of residuals is based on \emph{$s^*$-com\-pres\-si\-bi\-li\-ty} of operators \cite{CDD01}: an operator $\bA$ on $\ell_2(\N)$ is called \emph{$s^*$-compressible} with $s^*>0$ if for each $s \in (0, s^*)$, there exist operators $\bA_j$ and $\alpha_j>0$ for $j \in \N$ such that $\sum_{j} \alpha_j < \infty$, each $\bA_j$ has at most $\alpha_j 2^j$ nonzero entries in each row and column, and $\norm{\bA - \bA_j} \leq \alpha_j 2^{-s j}$. 
In order to approximate $\bA \bv$ for given $\bv$, taking $\bv_j$ to be the vectors retaining only the $2^j$ entries of $\bv$ of largest modulus, one then sets 
\begin{equation}\label{eq:sstarapprox}
 \bw_J = \bA_J \bv_0 + \sum_{j=1}^J 	 \bA_{J-j} (\bv_{j} - \bv_{j-1}),
\end{equation}
which amounts to assigning the most accurate sparse approximations of $\bA$ to the largest coefficients of $\bv$.
With $J$ chosen to ensure $\norm{ \bw_J - \bA \bv} \leq \eta$ for given $\eta$, as shown in \cite{CDD01}, evaluating this residual approximation requires $\mathcal{O}(\eta^{-1/s} \norm{\bv}_{\cA^s}+ \# \supp \bv + 1)$ operations. 
With this approximation used for step \textbf{(i)} in Algorithm \ref{AGM} with appropriately chosen parameters, from the results in \cite{GHS}, we obtain the following: if $\bu \in \cA^s$ for an $s < s^*$, the method yields a $\bu^k$ with $\norm{ \bB \bu^k - \bbf} \leq \varepsilon$ using $\mathcal{O}( 1 + \varepsilon^{-1/s} \norm{\bu}_{\cA^s})$ operations; that is, the method has optimal complexity for all $s < s^*$.

An adaptive scheme using wavelet approximation in space was constructed in \cite{BCD}, using the following observation that crucially depends on the multilevel property \eqref{multilevel2}.

\begin{prop}\label{BoundSemiDisc}
Let \eqref{multilevel2} hold. Then for $\ell \in \N_0$,
\[
\Bignorm{\bB - \sum_{\substack{\mu \in \cM_0 \\ \abs{\mu} <  \ell }  } \bM_\mu \otimes \bA_\mu} \leq  C_\bB 2^{-\ell\alpha},\quad \text{where $\displaystyle C_\bB := \frac{C_\Psi}{c_\Psi} \frac{C_2}{(1 - 2^{-\alpha})}$}
\]
with $\alpha$ and $C_2$ as in \eqref{multilevel2} and $c_\Psi, C_\Psi$ from \eqref{rieszV}.
\end{prop}

\begin{proof}
For $v, w \in \cV$, we have
\[
 \int_Y \int_D \sum_{\M{\abs{\mu}\geq \ell}} y_\mu \theta_\mu \nabla v(y)\cdot\nabla w(y)\sdd x\sdd \sigma(y)
 \leq \int_Y \int_D \sum_{\M{\abs{\mu}\geq \ell}} \abs{\theta_\mu} \abs{\nabla v(y)} \abs{\nabla w(y)}\sdd x\sdd \sigma(y),
\]
and the \M{right-hand side} is bounded by $C_2 ( 1 - 2^{-\alpha}) \M{2^{-\alpha \ell}}  \norm{v}_\cV \norm{w}_\cV$ as a consequence of \eqref{multilevel2}. With the orthonormality of the product Legendre polynomials and the bounds \eqref{rieszV} on the spatial Riesz basis, the statement follows.
\end{proof}

The above observation will also play a role in our new approach, which is presented in the following section. Let us now briefly review how it was used in the residual approximation analyzed in \cite{BCD}.
There, in order to obtain a fully discrete operator compression, the approximation provided by Proposition \ref{BoundSemiDisc} was combined with wavelet compression of the infinite matrices $\bA_\mu$. The following bounds show the dependence of their compressibility on $\mu$.

\begin{prop}[{see \cite[Prop.\ A.2]{BCD}}]\label{BCD_A2}
Let $\{ \theta_\mu\}_{\mu\in\cM_0}$ satisfy Assumptions \ref{ass:wavelettheta}, and for some $t > 0$, let
\begin{equation}\label{productsmoothness}
	   \theta_\mu \nabla\psi_{\lambda'}\in H^{t}(\supp \psi_\lambda), \quad \mu \in \cM_0, \; \lambda,\lambda'\in\cS, 
\end{equation}	   
and let the $\psi_\lambda$ have vanishing moments of order $k$ with $k>t-1$. 
Then there exist $\bA_{\mu,n}$ for $n\in\N$ such that the following holds:
\begin{enumerate}[{\rm(i)}]
\item
With $\tau \coloneqq t /d$, one has $
\|\bA_\mu - \bA_{\mu,n}\|\lesssim 2^{-\alpha\abs{\mu}-\tau n}$, $n\in\N$.
\item  The number of nonvanishing entries in each column of $\bA_{\mu,n}$ does not exceed $C \bigl(1+|\mu|^q\bigr)2^n$, where $q\coloneqq \max\{1, \tau^{-1} \}$ and $C>0$ is independent of $\mu,n$.
\end{enumerate}
\end{prop}

In \M{the} following abridged version of \cite[Prop.~4.3]{BCD}, with slightly sharpened assumptions, the two previous propositions are used to obtain $s^*$-compressibility of $\bB$.

\begin{cor}\label{cor:previous}
	Let $\{ \theta_\mu\}_{\mu\in\cM_0}$ satisfy Assumptions \ref{ass:wavelettheta}, and let $\Psi$ be as in Proposition \ref{BCD_A2} for some $t > \max\{\alpha-d,0\}$. 
	For any $L\in \N$, there exists a $\bC_L$ such that the following holds:
	
	\begin{enumerate}[{\rm(i)}]
		\item One has $
\|\bB - \bC_L \|\lesssim L  2^{-\alpha L}$.
		\item The number of nonvanishing entries in each column of $\bC_L$ does not exceed $C(1+L^q)2^{d(1+\tau^{-1})L}$, where  $q= \max\{1, \tau^{-1} \}$, $\tau=t/d$, and $C>0$ is independent of $L$.
	\end{enumerate} 
\end{cor}
\begin{proof}
	For $L\in \N$, take for any $\mu$ with $|\mu| < L$ an approximation $\bA_{\mu,n_\mu}$ as in Proposition \ref{BCD_A2}
	with $
	n_\mu = \left\lceil{\frac{d}{\tau}|\mu| + \frac{\alpha}{\tau} (L-|\mu|)} \right\rceil
	$.
	With this choice of $\bA_{\mu,n_\mu}$, let 
	\[
	\bC_L = \sum_{|\mu|<L} \bM_{\mu}\otimes \bA_{\mu,n_\mu}.
	\]
	Due to Proposition \ref{BoundSemiDisc}, we have 
	\[
	\|\bB-\bC_L\| \lesssim \sum_{|\mu|<L} \| \bM_{\mu}\otimes (\bA_{\mu} - \bA_{\mu,n_\mu})\| + 2^{-\alpha L}.
	\]
	By construction, for any $\mu$ with $|\mu|<L$ we have 
	\[
	\|\bA_{\mu} - \bA_{\mu,n_\mu}\|\lesssim 2^{-\alpha |\mu| - \tau n_{\mu}} \leq 2^{-\alpha |\mu| - d|\mu|-\alpha (L-|\mu|)} = 2^{-d|\mu|-\alpha L}.
	\]
	Using this inequality and $\norm{\bM_\mu} \leq 1$, we see that 
	\[
	\|\bB-\bC_L\| \lesssim \sum_{|\mu|<L} 2^{-d|\mu|-\alpha L} + 2^{-\alpha L} 
	  \lesssim  L2^{-\alpha L},
	\]
	which proves (i). 
	To prove (ii), we first note that by \M{Proposition \ref{BCD_A2}}, the number of nonvanishing entries in each column of $\bA_{\mu,n_\mu}$ does not exceed
	\[
	C\bigl(1+|\mu|^q\bigr)2^{n_\mu} \leq 2C(1+|\mu|^q) 2^{\frac{d}{\tau}|\mu| + \frac{\alpha}{\tau} (L-|\mu|)} ,
	\]
	where $C$ is independent of $\mu$.
	Since $\bM_\mu$ is diagonal or bidiagonal, it follows that  the number of nonvanishing entries in each column of $\bC_{L}$ does not exceed
	\[
	4C\sum_{|\mu|<L} (1+|\mu|^q) 2^{\frac{d}{\tau}|\mu| + \frac{\alpha}{\tau} (L-|\mu|)}
	\leq  
		4C L^q 2^{d(1+\tau^{-1})L} \sum_{\ell=0}^{L-1}  2^{ (\frac{\alpha}{\tau}-d-\frac{d}{\tau}) (L-\ell)}.
	\]
	Using that $\frac{\alpha}{\tau}-d-\frac{d}{\tau} = \frac{1}{\tau}(\alpha - t - d)<0$ concludes the proof of (ii). 
\end{proof}

\begin{remark}
	The approximations $\bC_L$ for $L \in \N$ can be applied in compressed operator application based on $s^*$-compressibility as in \eqref{eq:sstarapprox}, as carried out in \cite{BCD}. 
	Using the residual approximation according to Corollary \ref{cor:previous} in the adaptive Galerkin scheme, by the main result of \cite{GHS} we then have the following: ensuring $\norm{\bu - \bu^k}\leq \varepsilon$ requires at most 
	\begin{equation*}
	 \text{$\mathcal{O}\bigl(1 + \varepsilon^{-\frac1s} \norm{\bu}_{\cA^s}^{\frac1s}\bigr)$ operations for any $s < \frac{t}{t + d} \frac{\alpha}{d}$,}
	\end{equation*}
	with $t$ as in \eqref{productsmoothness}. Compared to the approximability \eqref{eq:fullapproxAs} of the solution $\bu$, this means that the performance of the method is limited by the compression of the operator $\bB$. In other words, for the best approximation rates that would be achievable for the solution, the method is not optimal. However, if $t$ in the regularity condition \eqref{productsmoothness} is large, rates that are close to optimal can be achieved. As discussed in \cite[\S4.2]{BCD}, that this is feasible is tied to the multilevel structure of the functions $\theta_\mu$.
\end{remark}

The previous results from \cite{BCD} thus show that by exploiting multilevel expansions of random fields, adaptive methods can in principle come close to achieving optimality for such problems. However, the use of wavelet bases of very high regularity for the spatial discretizations can be difficult in practice.
The factor $t / (t + d)$ resulting from the spatial operator compression can be improved to some extent for piecewise smooth basis functions using results from \cite{S04}, but for $d\geq 2$, optimality is then still not achieved.
These limitations motivate the different approach to approximating residuals that we take in the following section.

\section{Tree-Based Residual Approximations}

\newcommand{\trem}[1]{{\operatorname{t}(#1)}}
\newcommand{\treme}[2]{{\operatorname{t}_{#2}(#1)}}
\newcommand{\Stree}{\mathrm{T}(\cS)}
\newcommand{\FStree}{\mathrm{T}^\cF\!(\cS)}
\newcommand{\treeAs}{\cA^s_\mathrm{t}}
\newcommand{\Thetatree}{\mathrm{T}(\Theta)}

In this section, we develop a new approach for performing the different steps of Algorithm \ref{AGM}. Its central component is a new residual approximation using piecewise polynomial basis functions and wavelet index sets with tree structure, where we rely on techniques developed in \cite{S14,RS}. Selecting the residual coefficients of largest absolute value under this tree constraint can then be realized by the quasi-optimal tree coarsening procedure from \cite{BD:04,B:18}.

We require some auxiliary results on tree approximation from \cite{CDD:03,S14}, where we use the following basic notions as defined in \cite{S14} \M{for the wavelet-type basis $\Psi$ as introduced in Section \ref{sec:fullseqspace}.}

\begin{definition}\label{def:tree}
To each $\lambda \in \cS$ with $\abs{\lambda}>0$, we associate a $\lambda'\in\cS$ with $\abs{\lambda'} = \abs{\lambda} -1$ and $\meas(\supp\psi_\lambda\cap \supp \psi_{\lambda'}) > 0$. We then call $\lambda$ a \emph{child} of the \emph{parent} $\lambda'$ and write $\mathrm{C}(\lambda')$ for the set of all children of $\lambda'$, where we assume $\max_{\lambda \in \cS} \# \mathrm{C}(\lambda) < \infty$.
We call a subset $S \subseteq \cS$ a \emph{tree} if 
$S$ contains all $\lambda \in \cS$ with $\abs{\lambda} = 0$ and 
for all $\lambda \in S$, if $\lambda \in \mathrm{C}( \lambda' )$ then also $\lambda' \in S$. 
We denote the set of subsets of $\cS$ having such tree structure by $\Stree$.
\end{definition}

In addition, we denote that $\lambda$ is a \emph{descendant} of $\lambda'$ in the tree structure (that is, there exists $K\in \N$ such that with $\lambda_0 = \lambda$ and $\lambda_K = \lambda'$, one has $\lambda_0 \in \mathrm{C}(\lambda_1)$, $\ldots$, $\lambda_{K-1} \in \mathrm{C}(\lambda_K)$) by $\lambda \prec \lambda'$, and that $\lambda$ is a descendant of or equal to $\lambda'$ by $\lambda \preceq \lambda'$. 

Approximability of $\bv \in \ell_2(\cS)$ by expansions with this tree structure is then quantified \M{similarly to \eqref{stdAsdef1} and \eqref{stdAsdef2}},
\begin{equation}
\label{treeAsdef}
  \norm{\bv}_{\treeAs} := \sup_{N\in\N_0} (N+1)^s
  \inf_{\M{\substack{ \supp\bw \subseteq S \in \Stree\\ \# S\leq N}}} \norm{\bv - \bw}_{\ell_2(\cS)}.
\end{equation}
In addition, for index sets in $\cF \times \cS$ where each spatial component has tree structure, we write
\begin{equation}\label{FStreedef}
 \FStree :=  \bigl\{ \Lambda \subseteq \cF\times\cS\colon \text{for all $\nu\in\cF$, }  \{ \lambda \in \cS \colon ( \nu,\lambda) \in\Lambda \} \in \Stree  \bigr\}.
\end{equation}

\subsection{Tree approximability}

For quantifying the sparsity of sequences in $\ell_2(\cS)$ under the additional tree structure constraint, as in \cite{CDD:03}, we use the following notion:
for $\M{\bv} \in \ell_2(\cS)$, define $\trem{\M{\bv}} = ( \treme{\M{\bv}}{\lambda} )_{\lambda\in \cS}  \in \ell_2(\cS)$ by
\begin{equation}\label{eq:tcfdef}
    \treme{\M{\bv}}{\lambda} := \Bigl( \sum_{\substack{\lambda' \in \cS\\ \lambda'\preceq \lambda}} \abs{\M{\bv}_{\lambda'}}^2\Bigr)^{1/2}.
\end{equation}
Note that for $S \in \Stree$, we then have
\[
   \norm{ \M{\bv} - \bP{S} \M{\bv} }^2 = \sum_{ \substack{ \lambda \in \cS \setminus S \\ \exists \mu \in S\colon \lambda \in \mathrm{C}(\mu) }} \bigabs{ \treme{\M{\bv}}{\lambda}}^2 \,,
\]
where $\bP{S}\M{\bv}$ is defined by $(\bP{S}\M{\bv})_\lambda = \M{\bv}_\lambda$ for $\lambda \in S$ and $(\bP{S}\M{\bv})_\lambda = 0$ otherwise.
We have the following criterion for membership of $\M{\bv}$ in $\treeAs$ in terms of $\trem{\M{\bv}}$.

\begin{prop}[\!\!{\cite[Prop.~2.2]{CDD:03}}]
\label{prop:treeapproxspace}
	If $p \in (0,2)$ and $\trem{\M{\bv}} \in \wellp$, then $\M{\bv} \in \treeAs$ with $s = \frac{1}{p} - \frac{1}{2}$ and $\norm{\M{\bv}}_{\treeAs} \lesssim \norm{\trem{\M{\bv}}}_{\wellp}$.
\end{prop} 

For our present purposes, we next show that the approximability result \eqref{fullapprox1} from \cite[Section 8.2]{BCDS} also holds in the more restrictive case of tree approximation using index sets from $\FStree$.

\begin{prop}\label{prop:treeapprox}
Under the assumptions of Corollary \ref{cor:fullapprox} \M{for $\bu$ as in \eqref{fullydiscreteexpansion}},
\begin{equation}\label{fullapproxtree}
   \norm{ \bu }_{\mathrm{t},p} :=  \biggl( \sum_{\nu \in \cF} \norm{ \trem{\bu_\nu} }_{ \wellp }^p \biggr)^{1/p} < \infty\quad \text{for any $p>0$ such that $\displaystyle \frac1p < \frac{\alpha}d + \frac12$.}
\end{equation}
\end{prop}

\begin{proof}
For the space $Z_{\hat\alpha}$ in Theorem \ref{thm:W2tau}, we have (using Rychkov's universal extension operator \cite{R:99}, see \cite[Thm.~14.3.1]{Ag:15}) a characterization as a Bessel potential space,
\[
  [ H^1(D) , W^2_\tau(D)]_{\hat\alpha} = H^{1+\hat\alpha}_r(D), \quad \frac{1}{r} = \frac12 + \left( \frac{1}{\tau} - \frac12 \right) \hat\alpha.
\]
For any $\beta \in (0,\hat\alpha)$, we have that $H^{1+\hat\alpha}_r (D)$ \M{is continuously embedded into the Besov space} $B^{1+\beta}_{r,r}(D)$.
As a consequence of \cite[Cor.~4.2]{CDDD:01} and \cite[Remark 2.3]{CDD:03}, for $v = \sum_{\lambda \in \cS} \bv_\lambda \psi_\lambda \in B^{1+\beta}_{r,r}(D)$, if
\begin{equation}\label{thetacond}
  \frac{\beta}{d} > \frac{1}{r} - \frac{1}{2}
\end{equation}
then one has
\[
    \norm{\trem{\bv}}_{\well{\hat{p}}} \sim  \norm{ \trem{\bv} }_{\cA^{\beta/d}} \lesssim \norm{ v }_{B^{1+\beta}_{r,r}(D)}, \quad \frac{1}{\hat p} = \frac{\beta}{d} + \frac12.
\] 
We can choose $\beta < \hat\alpha$ such that \eqref{thetacond} is satisfied if $\frac{1}{\tau} - \frac12 < \frac{1}{d}$. We thus have
\[
   \sum_{\nu \in \cF} \norm{(\bu_{\nu,\lambda})_{\lambda \in \cS} }_{\well{\hat{p}}}^p < \infty
\]
for any $p$, $\hat p$ such that
\[
\frac{1}{p} < \frac{\alpha}{d} + \frac12 
   + \left( \frac{1}{\tau} - \frac12 - \frac{1}{d} \right) \hat\alpha,
\qquad
   \frac{1}{\hat p} < \frac{\hat\alpha}{d} + \frac12,
\]
and with Theorem \ref{thm:W2tau}, we obtain the assertion by taking $\hat\alpha$ sufficiently close to $\alpha$ and taking $\tau >1$ such that $\frac{1}{\tau}$ is sufficiently \M{close to} $\frac12 + \frac{1}{d}$, where we make use of our assumption $d>1$.
\end{proof}

\M{In what follows, for $\bv \in \ell_2(\cF\times\cS)$, we denote by $\treesupp \bv$ the set $\Lambda \in \FStree$ with minimal $\#\Lambda$ such that $\supp \bv \subseteq \Lambda$.}
Balancing the spatial approximations for each Legendre coefficient as in \cite[Thm.~3.1]{BCDS}, the summability property \eqref{fullapproxtree} combined with Proposition~\ref{prop:treeapproxspace} yields the following result on best approximations with spatial tree structure.

\begin{cor}
\label{cor:treebestapproxnorm}
 Let $\bu \in \ell_2(\cF\times\cS)$ and let $p>0$ be such that $\norm{ \bu }_{\mathrm{t},p} < \infty$. Then there exists $C>0$ independent of $\bu$ such that for all $n\in\N$,
 \begin{equation}\label{treeapproximability1}
   \min \{  \norm{\bu - \bv} \colon \M{ \#\treesupp \bv } \leq n \}  \leq C n^{-\frac1p + \frac12} \norm{\bu}_{\mathrm{t},p}.
 \end{equation}
\end{cor}

Note that under the assumptions of Corollary~\ref{cor:fullapprox}, \eqref{treeapproximability1} and Proposition~\ref{prop:treeapprox} imply that for any $\varepsilon>0$, the smallest $\Lambda \in \FStree$ such that $\norm{\bu - \bv}\leq \varepsilon$ for a $\bv \in \ell_2(\cF\times\cS)$ with $\supp \bv \subseteq \Lambda$ satisfies
\begin{equation}\label{treeapproximability2}
  \#\Lambda \leq C_s\, \varepsilon^{-\frac1s} \norm{ \bu }_{\mathrm{t},p}^{\frac1s}, \quad\text{for any } s < \frac{\alpha}{d} \;\;\text{ and }\; \frac1p =  s + \frac12,
\end{equation}
where $C_s>0$ depends on $s$. In other words, using tree approximation in space we recover the same convergence rates up to $\frac{\alpha}{d}$ as without the tree constraint in \eqref{eq:fullapproxAs}.

\begin{remark}\label{rem:treenormequiv}
For $\frac1p = s + \frac12$, for any $v \in L_2(Y,V,\sigma)$ with fully discrete representation $\bv \in \ell_2(\cF\times \cS)$,
\begin{equation}
  \bignorm{\bigl( \norm{\bv_\nu} \bigr)_{\M{\nu \in \cF}} }_{\cA^s} \lesssim \norm{\bv}_{ \cA^s} \sim \norm{\bv}_{\wellp} \leq  \norm{ \bv }_{\mathrm{t},p} .
\end{equation}
\end{remark}

\subsection{Multi-indices of unbounded length}\label{sec:multiindices}

In the numerical scheme that we consider, the vectors $\bv_\nu$ for given $\bv \in \ell_2(\cF\times\cS)$ need to be accessed by indices $\nu \in \cF \subset \N_0^\cM$ that may have non-zero entries in arbitrary positions.
As a consequence of the bidiagonal structure of the matrices $\bM_\mu$, one needs to store and iterate over finite subsets $F \subset \cF$ and to be able to access vector elements indexed by any $\nu \in F$ as well as by the indices $\nu \pm e_\mu$ that differ in only one component.

In the class of problems under consideration, the indices $\nu \in \cF$ activated in near-best approximations are generally extremely sparse, that is, for many such indices $\nu$ one has
\[
  \#\supp \nu\, \ll \,\dim(\nu):=\max\{ \mu \in \cM\colon \nu_\mu \neq 0\}\,.
  \]

\begin{remark}\label{rem:nuscaling}
As shown in \cite[Prop.~6.6]{BCD}, there are examples of problem data for \eqref{affinepde} such that the nonincreasing rearrangements of $(\norm{u_\nu}_V)_{\nu \in \cF}$ and $( \norm{ u_{e_\mu}}_V )_{\mu \in \cM}$ have the same asymptotic decay. In such a case, for a smallest $F_\varepsilon \subset \cF$ realizing the approximation of $u$ with error $\varepsilon>0$, one has $\# F_\varepsilon \sim \max_{\nu \in F_\varepsilon} \dim(\nu)$. Numerical tests (see \cite{BCD}) indicate that more generally, for the class of problems considered here, one has to expect $\max_{\nu \in F_\varepsilon} \dim(\nu) \gtrsim \varepsilon^{-t}$ for some $t>0$.
\end{remark}

For storing elements of $\cF$, we assume a fixed enumeration of the indices $\cM$, which reduces the problem to storing vectors with integer indices.
In view of Remark \ref{rem:nuscaling}, direct storage of the required $\nu$ in the form $(\nu_1, \nu_2, \ldots,  \nu_{\dim(\nu)})$ is too inefficient and will in general lead to a deterioration of the computational complexity of the method by some negative power of $\varepsilon$ as noted in Remark \ref{rem:nuscaling}. As an alternative, a sparse encoding of indices is suggested in \cite{GittelsonThesis}, where for $\nu$ with $\supp \nu = \{ i_1, \ldots, i_n\}$, the vectors $(i_1,\ldots, i_n)$ and $(\nu_{i_1},\ldots, \nu_{i_n})$ are stored. 

\begin{remark}\label{rem:coding}
A further alternative that is always at least as efficient as both direct or sparse storage is a \emph{run-length coding} of zeros in $\nu$, where a sequence of $m$ zeros is represented by an entry $-m$. 
More precisely, each $\nu \in \cF$ is encoded as a tuple $(m_1, m_2, \ldots, m_N)$ with $N\in \N$, where $m_i \in \Z \setminus\{0\}$ for $i=1,\ldots,N$, and where either $\nu_1 = m_1$ if $m_1>0$, or $\nu_1 = \ldots = \nu_{-m_1} = 0$ if $m_1 < 0$, in which case $\nu_{-m_1 + 1} = m_2 > 0$; the further entries of $\nu$ are then given recursively by the same scheme.
For instance, the Kronecker vectors corresponding to the first coordinates are encoded as the tuples $(1), (-1,1), (-2,1),\ldots$, respectively. Given such a storage scheme, the stored $\nu$ can be mapped to linear indices by hashing or tree data structures with (amortized) costs of order $\mathcal{O}(\#\supp \nu)$. 
\end{remark}

In view of these considerations, in what follows we assume the required operations on multi-indices $\nu$ to incur costs proportional to $\#\supp \nu$.

\subsection{Semidiscrete residuals}

As a first step in our adaptive scheme, we consider the approximation of residuals with only a parametric semidiscretization as in \eqref{semidiscreteform}, where each spatial component is still an element of the full function space $V$. 
Here we use adaptive operator compression to construct a routine \textsc{Apply} taking as input a tolerance $\eta>0$ and any $\bv \in \ell_2(\cF \times \cS)$ with $\#\suppF \bv < \infty$, where
 \[
   \suppF \bv := \{ \nu \in \cF\colon \supp \bv_\nu \neq \emptyset \},
 \]
and that produces a $\bw := \text{\textsc{Apply}}(\bv; \eta)$ such that $\norm{ \bB \bv  - \bw} \leq \eta$. In addition, both $\#\suppF \bw$ and the number of required products of the form $\bA_\mu \bv_\nu$ for $\mu\in \cM$, $\nu\in\cF$ satisfy quasi-optimal bounds with respect to $\eta$.

Here, the only approximation that needs to be performed on $\bB$ concerns the infinite summation, and the approximation $\bw$ is obtained by a suitable combination of the truncated operators 
\begin{equation}\label{Bell}
  \bB_\ell = \sum_{\substack{\mu \in \cM_0 \\ \abs{\mu} <  \ell }  } \bM_\mu \otimes \bA_\mu,  \quad \ell \in \N_0,
\end{equation}
where $\bB_0 = 0$.

A strategy for semidiscrete approximation of the stochastic residual has also been devised in \cite{Gittelson:13}. Here we use a different construction that is specifically adapted to the multilevel structure of the expansion \eqref{eq:affinecoeff} based on Proposition \ref{BoundSemiDisc}.
The semidiscrete scheme is summarized in Algorithm \ref{Apply_nosort}, with the result returned in a form that facilitates its subsequent use in a fully discrete residual evaluation. We next prove a complexity estimate for this scheme. In optimizing the choice of the $\ell_j$ in \eqref{eq:elljdef}, we follow \cite[Thm.~4.6]{Dijkema:09}.

\begin{algorithm}[t]
	\caption{$( M(\nu) )_{\nu \in F} = \text{\textsc{Apply}}(\bv; \eta)$, for $\#\suppF \bv < \infty$, $\eta>0$.}\label{Apply_nosort}
\begin{enumerate}[{\bf(i)}]
\item If $\norm{\bB} \norm{\bv} \leq \eta$, return \M{the empty tuple} with $F = \emptyset$;
otherwise, with $\bar J:=\ceil{\log_2 \#\suppF \bv}$, 
 for $j = 0,\ldots, \bar J$, 
  determine $F_j \subset \#\suppF \bv$
   such that $\# F_j \leq 2^{j}$ 
   and that $\bP{F_j\times \cS} \bv$ satisfies
\begin{equation}\label{eq:Apply_nosort_2j}\tag{\M{A4.1.1}}
    \norm{\bv- \bP{F_j\times \cS} \bv } \leq C \min_{\# \tilde F \leq 2^j} \norm{ \bv - \bP{\tilde F\times \cS} \bv }  
\end{equation}
for an absolute constant $C>0$.
Choose $J$ as the minimal integer such that
\[
  \delta := \norm{\bB} \norm{\bv - \bP{ F_J \times \cS}\bv} \leq \frac{\eta}{2}.
\]

\item With $\bd_0 :=\bP{F_0\times \cS} \bv$, $\bd_j := (\bP{F_j\times \cS} - \bP{F_{j-1}\times \cS} )\bv$, $j=1,\ldots, J$, and $N_j := \# F_j$, set 
\begin{equation}\label{eq:elljdef}\tag{\M{A4.1.2}}
	\ell_j = \left\lceil\alpha^{-1} \log_2 \biggl( \frac{C_\bB}{\eta - \delta} \biggl( \frac{\norm{\bd_j}}{N_j} \biggr)^{\frac{\alpha}{\alpha+d}} \Bigl( \sum_{i=0}^J \norm{\bd_i}^{\frac{d}{\alpha+d}} N_i^{-\frac{\alpha}{\alpha+d}} \Bigr) \biggr)\right\rceil,
	\quad j = 0, \ldots, J.
\end{equation}

\item With $\bw$ given by
  \begin{equation}\label{semidiscrwdef}\tag{\M{A4.1.3}}
    \bw = \sum_{j=0}^J  \bB_{\ell_j} \bd_j ,
  \end{equation}
  for each $\nu \in F:= \suppF \bw$, collect the sets $M(\nu) \subset \cM_0 \times \suppF \bv$ of minimal size such that
\begin{equation}\label{semidiscrwdefM}\tag{\M{A4.1.4}}
 \bw_\nu = \sum_{(\mu,\nu') \in M(\nu)}  (\bM_{\mu})_{\nu,\nu'} \,\bA_{\mu} \, \bv_{\nu'},\quad \nu \in F,
\end{equation}
and return $( M(\nu) )_{\nu \in F}$.
\end{enumerate}
\end{algorithm}

\begin{prop}\label{semidiscrapply}
	Let $s >0$ with $s< \frac{\alpha}{d}$, let $\bB$ be as in \eqref{seqform}, let $\bv$ satisfy $\#\suppF \bv < \infty$, and let $\bw$ be the approximation as in \eqref{semidiscrwdefM} of $\bB\bv$ given by Algorithm \ref{Apply_nosort}. Then $\norm{\bB\bv-\bw} \leq  \eta$, for $F = \supp_\cF \bw$ we have
    \begin{equation}\label{eq:Fest}
    \# F \leq    \sum_{\nu \in F} \# M(\nu)
     \lesssim  \sum_{j=0}^J 2^{d \ell_j} \# F_j   \lesssim \eta^{-\frac1s} \bignorm{\bigl(\norm{\bv_\nu} \bigr)_{\nu\in\cF}}_{\cA^s}^{\frac1s},
    \end{equation}
    and $\ell_j$ for $j=0,\ldots,J$ in \eqref{eq:elljdef} satisfy
    \begin{equation}\label{eq:elljest}
      \max_j \ell_j \lesssim 1 + \abs{\log \eta} + \log \bignorm{\bigl(\norm{\bv_\nu} \bigr)_{\nu\in\cF}}_{\cA^s}.
    \end{equation}
  The constants in the inequalities depend on $C$ from \eqref{eq:Apply_nosort_2j}, $C_\bB$, $d$, $\alpha$, $s$, and on $C_1$ from \eqref{multilevel1}.
\end{prop}
\begin{proof}
With the notation of Algorithm \ref{Apply_nosort}, we first note that, because $\#M(\nu) >0$ for every $\nu \in F$,
\[
\begin{aligned}
  \# F \leq  \sum_{\nu \in F} \# M(\nu)\,.
\end{aligned}
\]
Since $\bM_\mu$ is diagonal or bi-diagonal for all $\mu$,
\[
\begin{aligned}
  \sum_{\nu \in \cF} \# M(\nu) & \leq \sum_{j=0}^J \sum_{\substack{\mu \in \cM_0 \\ \abs{\mu} \leq \ell_j}} \sum_{\nu \in F_j} \# \bigl\{ \nu' \in \suppF \bv \colon \bigl( \bM_\mu \bigr)_{\nu',\nu} \neq 0 \bigr\} 
    \\  & \leq \sum_{j=0}^J \sum_{\substack{\mu \in \cM_0 \\ \abs{\mu} \leq \ell_j}} 2 N_j \lesssim \sum_{j=0}^J 2^{d\ell_j} N_j =: T(\ell_0,\ldots,\ell_J) \,.
\end{aligned}
\]
Let $\tilde\ell_0,\ldots,\tilde\ell_J$ minimize $T(\tilde\ell_0,\ldots,\tilde\ell_J)$ subject to the constraint $\sum_{j=0}^J \norm{\bB - \bB_{\tilde\ell_j} } \norm{\bd_j} \leq \eta - \delta$. Then the choice \eqref{eq:elljdef} of $\ell_0,\ldots,\ell_J$ (which corresponds to performing this minimization over $\R^{J+1}$ and rounding to the next largest integer) ensures that $T(\ell_0,\ldots,\ell_J) \lesssim T(\tilde\ell_0,\ldots,\tilde\ell_J)$.

It thus remains to show that $T(\tilde\ell_0,\ldots,\tilde\ell_J) \lesssim \eta^{-\frac1s} A(\bv)^{\frac1s}$
with 
\[
   A(\bv) := \bignorm{\bigl(\norm{\bv_\nu} \bigr)_{\nu\in\cF}}_{\cA^s}^{\frac1s} \,.
\]
Since $\suppF \bv$ is bounded, we have $(\norm{\bv_\nu} \bigr)_{\nu\in\cF} \in \cA^s(\cF)$ and thus, for $j=0,\ldots,J$,
\[
 \|\bv- \bP{F_j\times \cS} \bv\|\leq C 2^{-sj} A(\bv) ,
\] 
which for $j = 1,\ldots,J$ yields
	\begin{equation}\label{djest}
   \norm{\bd_j} \leq \norm{ \bv - \bP{F_j\times \cS}\bv} + \norm{\bv - \bP{F_{j-1}\times \cS} \bv } \leq ( 1+ 2^s) A(\bv) 2^{-sj}.
\end{equation}
We now choose $s_1,s_2>0$ such that $s < s_1 < s_2 < \frac\alpha{d}$. Take $K \in \N$ with minimal $K \geq J$ such that 
$
  \sum_{j=0}^J 2^{-(K-j) s_1} \norm{\bd_j} \leq \eta - \delta
$.
Then 
\[
  \frac\eta{2} \leq  \eta - \delta <  \sum_{j=0}^J 2^{-(K-1-j) s_1} \norm{\bd_j} \lesssim \sum_{j=0}^J 2^{-(K-j) s_1 } 2^{-sj} A(\bv) \lesssim 2^{-K s} A(\bv) \,,
\]
which implies $2^K \lesssim \eta^{-\frac1s}A(\bv)^{\frac1s}$.
For each $j$, let $\hat\ell_j \in \Z$ be the smallest integers such that $C_\bB 2^{-\alpha\hat\ell_j} \leq 2^{-(K-j)s_1}$. Then on the one hand, by Proposition \ref{BoundSemiDisc} and the choice of $K$,
\[
 \norm{\bB \bv - \bw} \leq \sum_{j=0}^J \norm{\bB-\bB_{\hat\ell_j} }\norm{\bd_j} + \delta \leq 
 \sum_{j=0}^J  C_\bB 2^{-\alpha\hat\ell_j} \norm{\bd_j} + \delta \leq  \eta.
\]
On the other hand, using $s_2 d < \alpha$,
\[
  2^{-(K-j)s_1} < C_\bB 2^{-\alpha (\hat \ell_j -1)} \leq C_\bB 2^{- s_2 d( \hat\ell_j -1) },
\]
and as a consequence $2^{d \hat \ell_j} \lesssim 2^{(K-j) s_1 / s_2}$.
We thus obtain
\[
  T(\hat \ell_0,\ldots,\hat\ell_J) \lesssim \sum_{j=0}^J 2^{(K-j) s_1/s_2} 2^j \lesssim 2^K \lesssim \eta^{-\frac1s} A(\bv)^{\frac1s},
\]
completing the proof of \eqref{eq:Fest}. The estimate \eqref{eq:elljest} follows from \eqref{eq:elljdef} and \eqref{djest}. 
\end{proof}

\subsection{Fully discrete residual approximation using tree evaluation}

For the approximation of the full residual on $\cF\times\cS$, we use concepts developed in \cite{S14} and \cite{RS} for handling the spatial degrees of freedom. This requires $\psi_\lambda$, $\lambda \in \cS$, and $\theta_\mu$, $\mu \in \cM$, to be piecewise polynomial \M{functions}.

We assume a family of tesselations into open convex polygonal subsets of the spatial domain $D$ to be given,
resulting from a fixed hierarchy of refinements $\cT_1,\cT_2,\ldots$ of an initial tessellation $\cT_0$.
For each $j$, we assume the elements $T \in \cT_j$ to form a partition of $D$, that is,
$\bigcup_{T\in {\cT}_j } \overline{T} = \overline{D}$ and for $T_1,T_2\in \cT_j$ with $T_1\neq T_2$ we have $T_1\cap T_2 = \emptyset$, and $\meas(T) \sim 2^{-j}$ for $T \in \cT_j$.
Furthermore $ \cT_{j}$ is a refinement of  $ \cT_{j-1}$ in the sense that for any $T\in  \cT_{j-1}$, there exists a unique subset $\tau \subseteq \cT_{j}$ such that $\overline{T} = \bigcup_{T'\in \tau} \overline{T'}$, where $\#\tau$ is bounded independently of $j$ and $T$.
Conversely, for $j' < j$, there exists a unique element $T' \in \cT_{j'}$ such that $T\cap T' \neq \emptyset$. 
Let 
\begin{equation*}
\hat \cT := \bigcup_{j=0}^\infty \cT_j \,.	
\end{equation*}
We define a \emph{tiling} to be a finite subset $\cT\subset \hat \cT$ such that 
$\bigcup_{T \in \cT} \overline{T} = \overline{D}$ and the elements of $\cT$ are pairwise disjoint.
For each tiling $\cT$ and $m\in \N_0$, we write $\PP_m(\cT)$ for the set of $f \in L_2(D)$ that are piecewise polynomial \M{functions} of degree $m$ with respect to $\cT$, that is,
\[
  f = \sum_{T\in \cT} q_{T} \Chi_T
\]  
with \M{polynomial functions} $q_{T}$ of degree at most $m$.
If $v \in \PP_m(\cT_j)$ with sufficiently large $m$ and $j$, we denote by $\cT(v)$ the smallest tiling such that $v$ is \M{a piecewise polynomial function} on $\cT(v)$ and \M{define}
\[
  \M{\cT_{\neq 0}(v) := \{ T \in \cT(v) \colon v|_T \neq 0  \};}
\]
\M{in other words, $\cT_{\neq 0}(v)$ comprises those elements of $\cT(v)$ that are contained in $\supp v$.}

\begin{example}
In our numerical tests, we use dyadic subdivisions of the cube $D=[0,1]^d$, where for $j\geq 0$,
\[
\cT_j = \bigl\{  (2^{-j}(k_1-1), 2^{-j}k_1) \times \cdots\times (2^{-j}(k_d-1), 2^{-j}k_d)  \colon  k \in \{1,\ldots, 2^j\}^d  \bigr\} .
\]
Note that $\# \cT_j = 2^{dj}$; here, $\hat\cT$ is the set of dyadic subcubes of $D$.
\end{example}

In order to apply the results from \cite{S14,RS}, we make the following additional assumptions on our wavelet basis, which are satisfied for standard continuously differentiable spline wavelets.

\begin{ass}\label{ass:pwpolypsi}
 Let the wavelet-type Riesz basis $\Psi$ satisfy the following conditions:
\begin{enumerate}[{\rm(i)}]
\item $\diam\supp \psi_\lambda \sim 2^{-\abs{\lambda}}$ for $\lambda \in \cS$.
\item There exist $m \in \N$ and $k \in \N_0$ such that for all $\lambda \in \cS$, $\psi_\lambda \in H^2(D) \cap \PP_{m}(\cT_{i_\lambda})$ with $i_\lambda \leq \abs{\lambda} + k$, and $\#\cT_{\neq 0}(\psi_\lambda) \leq C$.
\item For each $\ell \in \N_0$, $\overline{D} = \bigcup_{\abs{\lambda} = \ell} \supp \psi_\lambda$.
\item For each $\lambda\in\cS$, if $\int_D \psi_\lambda \sdd x \neq 0$, then $\abs{\lambda} = 0$ or $\dist(\supp \psi_\lambda, \partial D) \lesssim 2^{-\abs{\lambda}}$.
\end{enumerate}
Moreover, we assume that for each $\ell \in \N_0$, there exist \M{a countable index set $\Sigma_\ell$ and} a single-scale basis $\Phi_\ell = \{\varphi_\lambda: \lambda\in \Sigma_\ell \}$ such that 
\begin{equation}\label{eq:scalingspan}
\linspan\{ \psi_\lambda \colon \abs{\lambda}\leq \ell\} = \linspan \Phi_\ell,
\end{equation}
 satisfying the following conditions:
\begin{enumerate}[{\rm(i)}]
\setcounter{enumi}{4}
\item $\diam \supp \varphi_\lambda \sim 2^{-\ell}$ for $\lambda \in \Sigma_\ell$.
\item For any $j$ and any $T\in \cT_j$, the functions $\varphi_\lambda|_{T}$ with $|\lambda| =j$, $\varphi_\lambda|_{T}\ne 0$ are linearly independent.

\end{enumerate}	
\end{ass}

For standard spline wavelet bases, a single-scale basis $\Phi_\ell$ satisfying the conditions is given by the scaling functions on level $\ell$.
For the single-scale index sets $\Sigma_\ell$, we again write $\abs{\lambda}=\ell$ for $\lambda \in \Sigma_\ell$. We assume without loss of generality that $\Sigma_\ell \cap \Sigma_{\ell'} = \emptyset$ for $\ell' \neq \ell$.
Note that due to \eqref{eq:scalingspan}, the conditions in Assumptions \ref{ass:pwpolypsi}(ii) also hold for the functions $\varphi_\lambda$, $\lambda \in \bigcup_{\ell \geq 0} \Sigma_\ell$. As a consequence of the locality conditions in Assumptions \ref{ass:pwpolypsi}(i) and (v), 
$\#\cT_{\neq 0}(\psi_\lambda)$ and $\#\cT_{\neq 0}(\varphi_\lambda)$ are uniformly bounded for all respective $\lambda$.

\begin{ass}\label{ass:pwpolytheta}
There exist $\tilde m \in \N$, $\tilde k \in \N_0$ such that for all $\mu \in \cM_0$, $\theta_\mu \in W^1_\infty(D) \cap \PP_{\tilde m}(\cT_{j_\mu})$ with $j_\mu \leq \abs{\mu} + \tilde k$, and $\#\cT_{\neq 0}(\theta_\mu) \leq \tilde C$.
\end{ass}

Assumptions \ref{ass:pwpolypsi} and \ref{ass:pwpolytheta} imply in particular that $\nabla\cdot(\theta_{\mu} \nabla \psi_{\lambda} )$ is \M{a piecewise polynomial function} on $\cT_{\neq 0}(\theta_\mu \psi_\lambda)$ with at most $\max\{\#\cT_{\neq 0}(\psi_\lambda), \#\cT_{\neq 0}(\theta_\mu) \} \leq C$ terms, where $C$ is a uniform constant. Note that with additional technical effort, one could also similarly treat more general $\theta_\mu$ that can be approximated (uniformly in $\mu$) by piecewise polynomials. This holds true, for instance, for the multilevel expansions of Gaussian random fields constructed in \cite{BCM:18}.

Following \cite[Def.~4.9]{S14}, we call $\hat S \in \Stree$ a \emph{graded tree} if 
 for any $\lambda\in \hat S$ and any $\lambda'\in \cS$ with $|\lambda'|=|\lambda|-1$ and $\meas(\supp\psi_{\lambda'}\cap \supp\psi_{\lambda})>0$ we have $\lambda'\in \hat S$.
Any finite $S \in \Stree$ can be extended to its smallest containing graded tree
as described in \cite[Alg.~4.10]{S14}.
  
For a given tiling $\cT$, we define the graded tree $\cS(\cT,\ell) \subset \cS$ containing all wavelet indices up to $\ell$ levels above each $T \in \cT$ 
as the smallest extension to a graded tree of 
 \[
  \bigl\{\lambda\in\cS \colon \text{ $\exists\,j\in \N_0, T\in \cT_j\colon$}  \meas(\supp \psi_\lambda \cap T) >0\,\wedge\, \abs{\lambda} \leq j+\ell\bigr\}  \,.
  \]
For the approximation of functionals induced by piecewise polynomial functions in $V'$, we then have the following result.

\begin{prop}[{see \cite[Lemma A.1]{RS}}]
\label{prop:relerror}
There exists $C = C(m) >0$ such that for any $\ell \in \N$ and any $f \in L_2(D)$ that is \M{a piecewise polynomial function} of degree $m$ with respect to a tiling $\cT\subset \hat\cT$, 
 \[
      \bignorm{ \bigl( f(\psi_\lambda)\bigr)_{ \lambda \in \cS \setminus \cS(\cT,\ell)} }_{\ell_2}
       \leq C 2^{-\ell} \bignorm{  \bigl( f(\psi_\lambda)\bigr)_{ \lambda \in \cS} }_{\ell_2} , \qquad\text{where}\quad  f(\psi_\lambda) := \int_D f\, \psi_\lambda  \dd x.
 \]
\end{prop}

For any finite graded tree $\hat S \in \Stree$, by \cite[Alg.~4.12]{S14}, we can construct $\Sigma(\hat S) \subset \bigcup_{\ell \geq 0} \Sigma_\ell$ such that
\begin{equation}\label{eq:singlescaledef}
  \linspan \{ \varphi_\lambda \colon \lambda \in \Sigma(\hat S) \} \supseteq \linspan \{ \psi_\lambda \colon \lambda \in \hat S \}	
\end{equation}
and the \emph{multi- to locally single-scale transformation} $\bT_{\hat S}$ such that, for any $\bv$ with $\supp \bv \subseteq \hat S$,
\begin{equation}\label{eq:m2s}
  	\sum_{\lambda \in \hat S} \bv_\lambda \psi_\lambda = \sum_{\lambda \in \Sigma(\hat S)} \bigl( \bT_{\hat S} \bv )_\lambda \varphi_\lambda\,.
\end{equation}
We use this transformation as follows: for given $r \in V'$ and a graded tree $\hat S$, to evaluate $\br = \bigl(r(\psi_\lambda)\bigr)_{\lambda\in \hat S}$, we first evaluate $\bs_\lambda =r(\varphi_\lambda)$ for $\lambda \in \Sigma(\hat S)$ and then obtain $\br = \bT_{\hat S}^\top \bs$, since \eqref{eq:m2s} implies $\langle \br,\bv\rangle = \langle \bv, \bT_{\hat S}^\top \bs\rangle$ for any $\bv$.

\begin{prop}\label{prop:transformcomplexity}
	For any given tiling $\cT$ and $\ell>0$, the number of operations required for
	building the graded tree $\hat S = \cS(\cT,\ell) \subset \cS$ and each subsequent application of $\bT_{\hat S}$ or its \M{transpose} $\bT_{\hat S}^\top$ to a vector is bounded by $C \#\hat S$, where $C>0$ depends only on $\Psi$ and $\{ \Phi_{\ell'} \colon \ell' \geq 0\}$; in particular, $\# \Sigma(\hat S) \lesssim \# \hat S \lesssim \# \cT$.
\end{prop}

\begin{proof}
The bound for the number of operations required for the extension to a graded tree is shown in \cite[Prop.~4.11]{S14}, the one for the application of the multi- to single-scale transform and its \M{transpose} in \cite[Prop. 4.14(a)]{S14}.
\end{proof}

\begin{algorithm}[t]
\caption{~~Transform representation of a piecewise polynomial function $v$ on $\tau \subset \hat\cT$ to the representation on the minimal tiling $\cT(v)$ for $v$.}
\label{tilingtransform}
\flushleft Given $v = \sum_{T \in \tau} p_T \Chi_T$ with polynomials $p_T$, where $\tau$ is a tree
\flushleft Initialize $\tilde \tau = \tau$, $\tilde p_T = p_T$
\flushleft For $j = 0, \ldots, \max\{ J\colon  \cT_J \cap \tau \neq \emptyset \}$,
\flushleft \quad for each $T \in \cT_j \cap \tilde\tau$,
\flushleft \quad \quad if $T$ has a child in $\tilde\tau \cap \cT_{j+1}$, 
\flushleft \quad \quad \quad with $T_1,\ldots,T_c$ being all children of $T$ in $\cT_{j+1}$,
\flushleft \quad \quad \quad replace $\tilde p_T$ by $\sum_{i=1}^c \tilde q_{T_i} \Chi_{T_i}$ with polynomials $q_{T_i}$
\flushleft \quad \quad \quad remove $T$ from $\tilde\tau$
\flushleft \quad \quad \quad for $i = 1,\ldots, c$,
\flushleft \quad \quad \quad \quad if $T_i  \in \tilde\tau$, 
\flushleft \quad \quad \quad \quad \quad then $\tilde p_{T_i} \gets \tilde p_{T_i} + \tilde q_{T_i}$ for $i = 1,\ldots, c$; 
\flushleft \quad \quad \quad \quad otherwise, 
\flushleft \quad \quad \quad \quad \quad add $T_i$ to $\tilde\tau$ and set $\tilde p_{T_i} = \tilde q_{T_i}$
\flushleft return $v = \sum_{T \in \tilde\tau} \tilde p_T \Chi_T$, where $\tilde\tau = \cT(v)$
\end{algorithm}

The basic scheme for residual approximation of the full residual on $\cF \times \cS$, \M{using Algorithm \ref{tilingtransform} as a subroutine,} is given in Algorithm \ref{optscheme}.
In the following analysis of this scheme, we use Assumptions \ref{ass:wavelettheta}, \ref{ass:pwpolypsi}, and \ref{ass:pwpolytheta}.
To simplify the exposition, we also assume $f$ to be piecewise polynomial with a uniform bound on $\#\cT(f)$.

\begin{algorithm}[htp]
\caption{~~$(\Lambda^+ ,\br, \eta, b) = \text{\textsc{ResApprox}}(\bv; \zeta, \eta_0, \varepsilon)$, for $\#\supp \bv < \infty$, relative tolerance $\zeta>0$, initial absolute tolerance $\eta_0$, target tolerance $\varepsilon$.}
\flushleft Let $\M{\treesupp \bv =} \{ (\nu,\lambda) \colon \nu \in F,\lambda \in S_\nu \}$ with $F \subset \cF$, $S_\nu\in \Stree$ and $v_\nu = \sum_{\lambda \in S_\nu} \mathbf{v}_{\nu,\lambda} \psi_\lambda$.
\flushleft Set $\eta = 2\eta_0$; choose $\hat \ell$ such that $\zeta_{\hat\ell} := C 2^{-\hat\ell} < \zeta$.
\begin{enumerate}[{\bf(i)}]
\item For each $\nu \in F$, transform $v_\nu$ to piecewise polynomials on tilings $\cT(v_\nu)$ by applying Algorithm \ref{tilingtransform}
  \item Set $\eta \gets \eta/2$
  \item Set $( M(\nu) )_{\nu \in F^+} = \APPLY( \bv; \eta)$ by Algorithm \ref{Apply_nosort}, such that the corresponding semi-discrete residual approximation is given by
  \begin{equation}\label{resapprox_nu}\tag{\M{A4.3.1}}
     \hat r_\nu := \delta_{0,\nu} f  -  \sum_{(\mu,\nu') \in M(\nu)}  (\bM_{\mu})_{\nu,\nu'} A_{\mu} v_{\nu'}  \quad \text{for each $\nu \in F^{+}$}
  \end{equation}
  
\item For each $\nu \in F^+$

\noindent
\quad Initialize $\hat r_\nu = \delta_{0,\nu} f$

\noindent
\quad For each $(\mu, \nu') \in M(\nu)$\vspace{6pt}

\noindent\vspace{6pt}
\quad\quad $\displaystyle  \hat r_\nu \gets \hat r_\nu -  (\bM_{\mu})_{\nu,\nu'}  \sum_{T \in \cT_{\neq 0}(A_\mu v_{\nu'})} A_\mu v_{\nu'}\big|_{T}$

\item For each $\nu \in F^+$, 
 use Algorithm \ref{tilingtransform} to transform the representation $\hat r_\nu = \sum_{\tilde T \in \tau_\nu} \tilde p_{\nu,\tilde T} \Chi_{\tilde T}$ from {\bf(iv)} with a tree subset $\tau_\nu \subset  \hat\cT$ to the representation $\hat r_\nu = \sum_{T \in \cT(\hat r_\nu)} p_{\nu,T} \Chi_T$ on the minimal tiling $\cT(\hat r_\nu)$
 
 \item For each $\nu \in F^+$
  set $S^{+}_\nu := \cS(\cT(\hat r_\nu), \hat\ell) \subset \cS$; Determine $\Phi_\nu = \{ \varphi_\lambda \}_{\lambda \in \Sigma_\nu}$ as the corresponding locally single-scale basis with $\linspan \Phi_\nu \supseteq \linspan \{ \psi_\lambda \}_{\lambda \in S^+_\nu}$, $\Sigma_\nu = \Sigma(S_\nu^+)$, according to \eqref{eq:singlescaledef}, evaluate the integrals
\[
  \bs_{\nu,\lambda} = \hat r_{\nu}(\varphi_\lambda) =  \int_D \hat r_\nu \varphi_\lambda\dd x \quad\text{for $\nu \in F^+$, $\lambda \in \Sigma_\nu$}
\]
and set $\br_\nu = \bT_{S^+_\nu}^\top \bs_\nu$

  \item Let $b = ( 1- \zeta_{\hat \ell})^{-1} \norm{\br} + \eta$. If $\eta \leq \frac{(\zeta - \zeta_{\hat\ell})}{ (1+\zeta)(1+\zeta_{\hat\ell})} \norm{\br}$ or $b \leq \varepsilon$, 
  
  \noindent
  \quad with $\Lambda^+=\{ (\nu,\lambda)\colon \nu \in F^+ , \lambda\in  S^{+}_\nu\}$, return $(\Lambda^+ ,\br, \eta, b)$; 
  
  \noindent otherwise, go to {\bf(ii)}
\end{enumerate}
\label{optscheme}
\end{algorithm}

\begin{lemma}\label{lmm:treecompletion}
	Let $v \in \linspan \{ \psi_\lambda \colon \lambda \in S\}$ with finite $S \in \Stree$. Then for each $\mu \in \cM_0$,
\[
  \#\cT(A_\mu v) \lesssim \M{ \#\bigl\{ T \in \cT(v) \colon  \text{ $T \subseteq T'$ for a $T' \in \cT_{\neq 0} (\theta_\mu) $} \bigr\} } + \#\cT_{\neq 0} (\theta_\mu) + \abs{\mu} \,.
\] 
\end{lemma}

\begin{proof}
	By our assumptions, both $v$ and $A_\mu v$ are piecewise \M{polynomials}, where $\supp A_\mu v \subseteq \supp \theta_\mu \cap \supp v$.
  \M{We next note that $\cT_{\neq 0}(A_\mu v)$ is obtained from $\cT_{\neq 0} (\theta_\mu)$ by possible refinements only within each $T \in \cT_{\neq 0}(\theta_\mu)$, and thus}
  \[
    \M{ \#\cT(A_\mu v) \leq \#\cT(\theta_\mu) + \#\bigl\{ T \in \cT(v) \colon  \text{ $T \subseteq T'$ for a $T' \in \cT_{\neq 0} (\theta_\mu)$} \bigr\} \,. }
  \]
	Since every element of $\hat \cT$ subdivides into a uniformly bounded number of children, we have $\#\cT(\theta_\mu) \lesssim \#\cT_{\neq 0} (\theta_\mu) + (\abs{\mu} + 1)$, where $\#\cT_{\neq 0} (\theta_\mu) \gtrsim 1$. 
\end{proof}

\begin{theorem}\label{thm:resapprox}
Let $(\Lambda^+, \br, \eta, b)$ be the return values of Algorithm \ref{optscheme}.
Then $\norm{\bB \bv - \bbf}\leq  b$ and either $b \leq \varepsilon$, or $\br$ satisfies 
\begin{equation}\label{eq:reserrbound} 
  \norm{\br - (\bbf - \bB \bv)} \leq \zeta \norm{\bbf - \bB \bv},
\end{equation}
where we have
$ \M{\#\treesupp\br} \leq \#\Lambda^+ = \sum_{\nu \in F^+} \# S^+_\nu$ with $S^+_\nu\in \Stree$ for each $\nu \in F^+$ and
\begin{multline}\label{rsuppest}
 \#\Lambda^+ \lesssim  \\
      \#\cT(f) +  \Bigl(   \eta^{-\frac1s} \bignorm{\bigl(\norm{\bv_\nu} \bigr)_{\nu\in\cF}}_{\cA^s}^{\frac1s} + \M{\#\treesupp \bv} \Bigr)  \Bigl( 1 + \abs{\log \eta} + \log \bignorm{\bigl(\norm{\bv_\nu} \bigr)_{\nu\in\cF}}_{\cA^s} \Bigr)\,.
\end{multline}
The number of operations required for computing $\br$ is bounded by a fixed multiple of
\begin{multline}\label{ropest}
   \bigl( 1 + \log_2 (\eta_0/\eta)\bigr) \biggl[ \#\cT(f) + \Bigl(   \eta^{-\frac1s} \bignorm{\bigl(\norm{\bv_\nu} \bigr)_{\nu\in\cF}}_{\cA^s}^{\frac1s} + \M{\#\treesupp \bv} \Bigr) \\ \times  \Bigl( 1 + \abs{\log \eta} + \log \bignorm{\bigl(\norm{\bv_\nu} \bigr)_{\nu\in\cF}}_{\cA^s} +\log \#\suppF \bv + \max_{\nu \in F^+} \#\supp \nu\Bigr) \biggr].
\end{multline}
\end{theorem}

\begin{proof}
We first show that the prescribed relative tolerance $\zeta$ is achieved. Define $\hat\br$ by $\hat\br_{\nu,\lambda} := \hat r_\nu(\psi_\lambda)$ for all $\lambda \in \cS$, and extend $\br$ to $\cS$ by setting $\br_{\nu,\lambda}=0$ for $\lambda \notin S^+_\nu$. With $\hat\ell$ sufficiently large, as a consequence of Proposition \ref{prop:relerror} applied for each $\nu$, one obtains any required relative error in \M{step {\bf(vi)}}. Thus, $\norm{\hat\br - \br} \leq \zeta_{\hat\ell} \norm{\hat\br}$. 
Algorithm \ref{Apply_nosort} ensures $\norm{\hat\br - (\bbf - \bB\bv)} \leq \eta$ whenever step {\bf(vii)} is reached. Thus if the algorithm stops due to the first condition in this step, by the triangle inequality, the error bound \eqref{eq:reserrbound} holds if $\eta + \zeta_{\hat\ell}\norm{\hat\br} \leq \zeta \norm{\bbf - \bB\bv}$. Since $\zeta  \norm{\bbf - \bB\bv} \geq \zeta (\norm{\hat\br} - \eta)$, a sufficient condition is $(1 + \zeta) \eta \leq (\zeta - \zeta_{\hat\ell}) \norm{\hat \br}$, and since moreover $\norm{\br} \leq ( 1 + \zeta_{\hat\ell}) \norm{\hat\br}$, this in turn is implied by the condition in the final step. If the algorithm stops due to the second condition in step {\bf(vii)}, then $\norm{\bB\bv - \bbf} \leq b \leq \varepsilon$.

By construction, there exists a $C_{\hat \ell} >0$ such that 
\[
  \sum_{\nu\in F^+} \# S^+_\nu \leq C_{\hat\ell} \sum_{\nu \in F^+} \# \cT(\hat r_\nu).
\]
Moreover, we have the upper bound
\[
 \# \cT(\hat r_\nu) \leq \delta_{0,\nu} \#\cT(f)  + \sum_{(\mu,\nu')\in M(\nu)} \#\cT( A_\mu v_{\nu'}) .
\]
With the corresponding $\ell_j$, $F_j$, and $\bd_j$ for $j=0,\ldots,J$ as in Algorithm \ref{Apply_nosort}, note first that by \eqref{semidiscrwdef}, we have
\[
\br = \bbf -  \sum_{j=0}^J \bB_{\ell_j} \bd_j 
  =  \bbf -  \sum_{j=0}^J  \sum_{k = 0}^{\ell_j-1} \sum_{\substack{\mu \in \cM_0 \\ \abs{\mu} = k }  } ( \bM_\mu \otimes \bA_\mu )  \bd_j \,.
\] 
Since $\bM_\mu$ is diagonal or bidiagonal for each $\mu$, 
\[
  \sum_{\nu\in F^+} \sum_{(\mu,\nu')\in M(\nu)} \#\cT( A_\mu v_{\nu'}) \leq 2
    \sum_{j=0}^J    \sum_{k = 0}^{\ell_j-1} \sum_{\nu \in F_j} \sum_{\substack{\mu \in \cM_0 \\ \abs{\mu} = k }  }\# \cT(A_\mu v_\nu).
\]
Since for each $\nu \in F$, the wavelet expansion of $v_\nu$ has tree structure by our assumption, Lemma \ref{lmm:treecompletion} yields
\[
  \M{  \#\cT(A_\mu v_\nu) \lesssim \#\bigl\{ T \in \cT(v_\nu) \colon  \text{ $T \subseteq T'$ for a $T' \in \cT_{\neq 0} (\theta_\mu)$} \bigr\} + \#\cT_{\neq 0} (\theta_\mu) + \abs{\mu}   }
\]
for each $\mu$ and $\nu$.
As a consequence of \M{Assumptions \ref{ass:wavelettheta}(i) and (ii), Assumptions \ref{ass:pwpolypsi}, \ref{ass:pwpolytheta} as well as \eqref{multilevel1},}
\[
  \M{     \sum_{\substack{\mu \in \cM_0 \\ \abs{\mu} = k }  }  \#\bigl\{ T \in \cT(v_\nu) \colon  \text{ $T \subseteq T'$ for a $T' \in \cT_{\neq 0} (\theta_\mu)$} \bigr\}  \lesssim   \# \cT(v_\nu)  ,}
 \]
 and moreover,
 \[
   \M{  \sum_{\substack{\mu \in \cM_0 \\ \abs{\mu} = k }  }\# \cT_{\neq 0}(\theta_\mu) \lesssim 2^{d k}, } \qquad     \sum_{\substack{\mu \in \cM_0 \\ \abs{\mu} = k }  } \abs{\mu} \lesssim k 2^{dk}.
\]
Since $S_\nu$ is a tree, $\cT(v_\nu) \lesssim \# S_\nu$ for all $\nu \in F$. 
Putting the above estimates together, we obtain
\[
\begin{aligned}
  \sum_{\nu\in F^+} \# S^+_\nu   &\lesssim   \#\cT(f) +   \sum_{j=0}^J  \sum_{\nu \in F_j}  \sum_{k = 0}^{\ell_j-1} \left((1+ k) 2^{dk} +  \#\cT(v_\nu) \right)  \\
    &  \lesssim    \#\cT(f) +  \sum_{j=0}^J \max\{\ell_j,0\} \biggl(  2^{d \ell_j}  \# F_j  + \sum_{\nu \in F_j} \# S_\nu   \biggr) \\
    & \leq \#\cT(f) + \bigl(\max_{j = 0,\ldots,J} \max\{ \ell_j , 0 \} \bigr)\biggl( \sum_{j=0}^J 2^{d \ell_j}  \# F_j  + \sum_{\nu \in F} \# S_\nu \biggr)  \,.
\end{aligned}
\]
With Proposition \ref{semidiscrapply}, and noting that $\sum_{\nu \in F} \#S_\nu \M{ = \#\treesupp \bv }$, we obtain \eqref{rsuppest}.

It remains to estimate the number of required operations.
Since $S_\nu$ is a tree for each $\nu\in F$, the number of operations for step {\bf(i)} of Algorithm \ref{optscheme} is bounded by a multiple of \M{$\#\treesupp \bv$}. 
For the computation of norms and sorting, $\APPLY$ in \M{step {\bf(iii)}} requires a number of operations bounded by a fixed multiple of
    \[
       \M{ \#\treesupp \bv } + \#\suppF \bv \,\log \#\suppF \bv\,.
    \] 
From Proposition \ref{semidiscrapply}, we have
\[
    \sum_{\nu \in F^+} \# M(\nu) \lesssim  \eta^{-\frac1s} \bignorm{\bigl( \norm{\bv_\nu} \bigr)_{\nu \in \cF}}_{\cA^s}^{\frac1s} .
    \]
    The number of operations for handling multi-indices in steps \textbf{(iii)} and \textbf{(iv)} is thus, according to Remark \ref{rem:coding}, bounded by a multiple of 
    \[
    \eta^{-\frac1s} \bignorm{\bigl( \norm{\bv_\nu} \bigr)_{\nu \in \cF}}_{\cA^s}^{\frac1s} \Bigl( 1 + \max_{\nu \in F^+} \#\supp \nu \Bigr).
    \]
     The further operations in steps {\bf(iv)} and {\bf(v)} combined require a number of operations bounded by a multiple of 
   \begin{multline*}
     \# \cT(f) +  \sum_{\nu\in F^+} \sum_{(\mu,\nu')\in M(\nu)}  \bigl( \#\bigl\{ T \in \cT(v_\nu) \colon  \text{$T \subseteq T'$ for a $T' \in \cT_{\neq 0} (\theta_\mu)$} \bigr\} + \#\cT_{\neq 0} (\theta_\mu)  + \abs{\mu}  \bigr) \\
       \lesssim \#\cT(f) + \bigl(\max_{j = 0,\ldots,J}  \max\{\ell_j,0\} \bigr)\biggl( \sum_{j=0}^J 2^{d \ell_j}  \# F_j  + \sum_{\nu \in F} \# S_\nu \biggr),
   \end{multline*}
   which we estimate further as above.
   Concerning step {\bf(vi)}, note that for each $\nu$ and $\lambda \in \Sigma_\nu$, the number of elements of $\cT(\hat r_\nu)$ intersecting $\supp \varphi_\lambda$ is uniformly bounded by construction of $\Sigma_\nu$, and thus the computation of each integral $\hat r_\nu(\varphi_\lambda)$ of piecewise polynomial functions on this tiling requires a uniformly bounded number of operations.  
    As a consequence of Proposition \ref{prop:transformcomplexity}, the required number of operations for step {\bf(vi)} is thus bounded by a fixed multiple of $\sum_{\nu \in F^+} \# \cT(\hat r_\nu)$.

   In summary, each execution of the body of the loop from steps {\bf(iii)} to {\bf(vi)} requires a number of operations bounded by a fixed multiple of 
  the upper bound in \eqref{rsuppest}
   with the current value of $\eta$. In terms of the value of $\eta$ that is returned, the number of iterations in the outer loop is bounded by $1+\log_2(\eta_0/\eta)$ times, which yields the bound \eqref{ropest} for the total number of operations.
\end{proof}

\subsection{Tree coarsening}
In step {\bf(ii)} of Algorithm \ref{AGM}, for a given residual approximation $\br$ of the Galerkin solution on $\Lambda^0 =  \{ (\nu,\lambda)\colon \nu \in F^0 , \lambda\in  S^{0}_\nu\}$, with $\supp \br \subseteq \Lambda^+ = \{ (\nu,\lambda)\colon \nu \in F^+ , \lambda\in  S^{+}_\nu\}$, with $0<\omega_0\leq  \omega_1 < 1$ we need to find $\Lambda^\flat \subseteq \Lambda^+$ with $\M{\Lambda^\flat} \in \FStree$ satisfying \eqref{eq:agmbulk}, that is,
\begin{equation}\label{bulktree}
   \norm{\br|_{\Lambda^\flat}} \geq \omega_0 \norm{\br}\quad  \text{and} \quad
  \#(\Lambda^\flat \setminus \Lambda^0) \leq \tilde C \#(\tilde\Lambda\setminus \Lambda^0) 
\end{equation}
with $\tilde C>0$ for any $\tilde\Lambda \supset \Lambda^0$, $\tilde \Lambda \in \FStree$, such that $\norm{\br|_{\tilde\Lambda} }\geq \omega_1 \norm{\br}$. 
 For finding such near-best $\Lambda^\flat$, and hence $S^\flat_\nu \subseteq S^+_\nu$, that additionally have tree structure, we follow the strategy of the \emph{thresholding second algorithm} from \cite{BD:04} (see also \cite{B:07}), in the version stated in \cite{B:18}.

To determine $\Lambda^\flat$, we use the tree structure of $\Lambda^+ \setminus \Lambda^0$ as follows. We can assume without loss of generality that $\lambda_0$ is the single root element of $\cS$; to this end, we can group all $\lambda\in\cS$ with $\abs{\lambda}=0$ into a single element of the tree by always adding these $\lambda$ jointly to an index set. This ensures that all generated spatial index sets are trees according to Definition \ref{def:tree}.
We thus also have a tree structure on $\cF\times \cS$: for $(\nu, \lambda) \in \cF\times\cS$, we write $\delta := (\nu, \lambda)$, where $\delta' = (\nu,\lambda')\prec (\nu,\lambda)  = \delta$ if and only if $\lambda' \prec \lambda$. The subsets of $\cF\times\cS$ that are trees with respect to each of their spatial components are then precisely the elements of $\FStree$ as in \eqref{FStreedef}, and are again referred to as trees.

For $\delta \in \cF\times\cS$, let $\Theta_\delta$ be the infinite subtree of $\cF\times\cS$ with root element $\delta$.
Since we aim to select elements of $\Lambda^+ \setminus \Lambda^0$, in the tree coarsening scheme we only operate on subtrees of $\Theta := (\cF \times \cS) \setminus \Lambda^0$ with root elements 
\begin{equation}\label{eq:Delta0def}
   \Delta_0 =  \bigl\{  (\nu , \lambda_0) \colon \nu \in F^+ \setminus F^0  \bigr\} \cup \bigl\{  (\nu,\lambda)\colon \nu \in F^0, \lambda \notin S_\nu^0\text{ with }\lambda \in \mathrm{C}( \lambda' ) \text{ for a $\lambda' \in S^0_\nu$} \bigr\} ,
\end{equation}
where $\Theta = \bigcup_{\delta \in \Delta_0} \Theta_\delta$. We accordingly introduce the set of tree subsets
\[
   \Thetatree := \{ \Lambda \subset \Theta \colon (\Lambda \ni \delta \prec \delta ' \in \Theta ) \implies \delta' \in \Lambda \}.
\]

For an arbitrary tree $\Lambda \in  \Thetatree$, we write $\mathrm{L}(\Lambda)$ for its \emph{leaves}, that is, for the elements of $\Lambda$ that do not have any child in $\Lambda$. Moreover, $\mathrm{I}(\Lambda) = \Lambda \setminus \mathrm{L}(\Lambda)$ are the \emph{internal nodes} of $\Lambda$. For $\delta \in \Lambda$, we again write $\mathrm{C}(\delta)$ for the set of all children of $\delta$ in $\Theta$, where $\max_{\delta \in \Theta} \# \mathrm{C}(\delta) < \infty$.
We call a tree $\Lambda \in \Thetatree$ \emph{proper} if $\mathrm{C}(\delta) \subset \Lambda$ for any $\delta \in \mathrm{I}(\Lambda)$.

Following \cite{BD:04}, for $\delta \in \Theta$ we define the error measures
\begin{equation}\label{errmeas}
   e(\delta) =  \begin{cases}  \abs{ \treme{\br_\nu}{\lambda}  }^2, & \text{for $\delta = (\nu,\lambda) \in \Lambda^+$}, \\
      0, &\text{for $\delta \in \Theta \setminus \Lambda^+$,} \end{cases}
\end{equation}
for which we have the subadditivity property
\begin{equation}\label{subadd}
  e(\delta) \geq \sum_{\delta' \in \mathrm{C}(\delta)} e(\delta')\,.
\end{equation}
In addition, for each $\Delta \in \Thetatree$ we define the global error measure
\[
  E(\Delta) = \sum_{\delta \in \mathrm{L}(\Delta)} e(\delta),
\]
which by \eqref{eq:tcfdef} satisfies $E(\Delta) = \norm{\bP{\Theta} \br - \bP{\mathrm{I}(\Delta)} \br }^2$,
and the corresponding best approximation errors
\[
  \sigma_N = \min_{ \substack{\Lambda \in \Thetatree \\ \# \mathrm{I}(\Lambda) \leq N } } E(\Lambda) \,.
\]
Note that since only interior nodes are counted, the trees realizing the best approximation can always be assumed to be proper trees.

The algorithm from \cite{B:07,B:18} is based on the \emph{modified errors} for $\delta \in \Theta$,
\begin{equation}\label{errmeasmod}
  \tilde e(\delta) = \begin{cases}  e(\delta), &\text{if $\delta \in \Delta_0$,} \\ \bigl( e(\delta)^{-1} + \tilde e(\delta^*)^{-1}\bigr)^{-1}, & \text{if $\delta^*$ is the parent of $\delta$.}  \end{cases}
\end{equation}
The greedy-type scheme producing the sought approximation is stated in Algorithm \ref{alg:tree}. The returned trees are proper trees by construction.

\begin{algorithm}[t]
	\caption{$\Lambda^\flat := \TREEAPPROX(\Lambda^0, \Lambda^+, \br, \eta)$}\label{alg:tree}
\flushleft Set $N:=0$ and 
$\Delta_0$ as in \eqref{eq:Delta0def}
\flushleft Evaluate $e(\delta)$ and $\tilde e(\delta)$ for $\delta\in \Lambda^+$ according to \eqref{errmeas}, \eqref{errmeasmod}
\flushleft Until $E(\Delta_N) \leq  \eta$, repeat
\flushleft \qquad find a $\delta \in \mathrm{L}(\Delta_{N})$ with largest $\tilde e(\delta)$
\flushleft \qquad set $\Delta_{N+1} = \Delta_N \cup \mathrm{C}(\delta)$ and $N\gets N+1$
\flushleft Return $\Lambda^\flat = \Lambda^0 \cup \mathrm{I}(\Delta_N)$
\end{algorithm}

The following two lemmas can be obtained by minor modifications of \cite[Lemmas 2.3, 2.4]{B:18}, where the analogous statements are shown for binary trees with a single root element. For the convenience of the reader, we give the proofs in Appendix \ref{app:trees}.
\begin{lemma}\label{lmm:treeupper}
Let $\eta>0$, and let $\Lambda$ be a finite tree such that $\tilde e(\delta) \leq \eta$ for all $\delta \in \mathrm{L}(\Lambda)$. Then 
\[
 \sum_{\delta \in \mathrm{L}(\Lambda)} e(\delta) \leq (\# \Lambda) \eta\,.
\]
\end{lemma}

\begin{lemma}\label{lmm:treelower}
 Let $\eta>0$, let $\delta_0$ be a node in a finite tree $\Lambda$, and let $\Lambda_{\delta_0}$ be a subtree of $\Lambda$ rooted at $\delta_0$ such that $\tilde e(\delta) \geq \eta$ for all $\delta \in \Lambda_{\delta_0}$. Then
 \[
    e(\delta_0)  \geq (\#\Lambda_{\delta_0}) \eta\,.
 \]
\end{lemma}

With these lemmas at hand, we obtain the following modification of \cite[Thm.~2.1]{B:18} for our setting.

\begin{theorem}\label{thm:treecoarsen}
 Let $\Delta_0,\dots,\Delta_N$ be as constructed in Algorithm \ref{alg:tree}. Then we have
 \begin{equation}\label{eq:treecoarsen}
    E(\Delta_k) \leq \frac{k+1}{k-n+1} \sigma_n, \quad 0 \leq n \leq k \leq N.
 \end{equation}
\end{theorem}

\begin{proof}
The statement clearly holds if $k=0$ or $n=0$, and we can thus assume $k,n\geq 1$.
Let $\Delta^*_n$ be a tree realizing the best approximation for $n$, so that $E(\Delta^*_n) = \sigma_n$. If $\mathrm{I}(\Delta^*_n) \subseteq \mathrm{I}(\Delta_k)$, then $\Delta^*_n\subseteq \Delta_k$ and thus $E(\Delta_k) \leq E(\Delta^*_n)$. Otherwise, there exists an element of $\mathrm{I}(\Delta^*_n)$ that is not in $\mathrm{I}(\Delta_k)$. We now estimate $E(\Delta^*_n)$ from below in terms of
\[
  \tilde m_k := \max_{\delta \in \mathrm{L}(\Delta_k)} \tilde e(\delta) \,.
\]
Let $D := \mathrm{I}(\Delta_k) \setminus \mathrm{I}(\Delta^*_n)$. Since there is at least one node from $\mathrm{I}(\Delta^*_n)$ that is not in $\mathrm{I}(\Delta_k)$, 
we have $\# D = k - n + 1$.
Note that $D$ is the union of the trees $\Theta_\delta \cap \mathrm{I}(\Delta_k)$ for $\delta \in \mathrm{L}(\Delta^*_n)$. By Lemma \ref{lmm:treelower}, 
\begin{equation}\label{sigmanupper}
 \begin{aligned}
  \sigma_n &= E(\Delta^*_n) = \sum_{\delta \in \mathrm{L}(\Delta^*_n)} e(\delta) \\
  & \geq  \sum_{\delta \in \mathrm{L}(\Delta^*_n)}  (\# \Theta_\delta \cap \mathrm{I}(\Delta_k))\,\tilde m_k \geq \# D\, \tilde m_k \geq (k-n+1) \tilde m_k.
 \end{aligned}
\end{equation}
In order to estimate $E(\Delta_k)$ from above by $\tilde m_k$, we note that
\[
    E(\Delta_k) = \sum_{\delta \in \mathrm{L}(\Delta_k) \setminus \mathrm{I}(\Delta^*_n)} e(\delta) +  \sum_{\delta \in \mathrm{L}(\Delta_k) \cap \mathrm{I}(\Delta^*_n)} e(\delta) ,
\]
where on the one hand
\[
  \sum_{\delta \in \mathrm{L}(\Delta_k) \setminus \mathrm{I}(\Delta^*_n)} e(\delta) \leq  \sum_{\delta \in  \mathrm{L}(\Delta^*_n)} e(\delta) = \sigma_n,
\]
and on the other hand, applying Lemma \ref{lmm:treeupper} to the minimal tree with leaves $\mathrm{L}(\Delta_k) \cap \mathrm{I}(\Delta^*_n)$,
\[
   \sum_{\delta \in \mathrm{L}(\Delta_k) \cap \mathrm{I}(\Delta^*_n)} e(\delta) \leq 
    ( \#  \mathrm{I}(\Delta^*_n) ) \,\tilde m_k = n \tilde m_k.
\]
Combining these bounds with \eqref{sigmanupper}, we obtain
\[
  E(\Delta_k) \leq \sigma_n + \frac{n}{k-n+1} \sigma_n = \frac{k+1}{k-n+1} \sigma_n,
\]
which completes the proof.
\end{proof}

As in \cite{B:07}, from Theorem \ref{thm:treecoarsen} we obtain the following variant of \cite[Cor.~5.4]{BD:04}.

\begin{cor}\label{cor:treecoarsencard}
With $\Delta_N$ as generated by Algorithm \ref{alg:tree}, for any $c \in (0,1)$ and any proper tree $\tilde\Delta$ with roots $\Delta_0$ such that $E(\tilde\Delta) \leq c \eta$, we have
\[
  \# \mathrm{I}(\Delta_N) \leq \tilde C \# \mathrm{I}(\tilde\Delta),
\]
with $\tilde C>0$ depending only on $c$.
\end{cor}

\begin{proof}
Let $\Delta_0,\dots, \Delta_N$ be as constructed by Algorithm \ref{alg:tree}. 
By construction, we have $E(\Delta_0) = \sigma_0$ and $E(\Delta_1) = \sigma_1$, and we can thus assume $N > 1$. Let $\gamma =  1 - c$, then $\gamma(N-1) \leq (1-c)N$ and hence 
\begin{equation}\label{eq:tccard}
  \frac{N}{N - n} \leq  \frac{N}{(1-\gamma) N + \gamma} \leq c^{-1} , \quad 0 \leq n \leq \gamma (N-1),
\end{equation}
for all such $N$. Let $\tilde\Delta^*$ be a proper tree with minimal $n^* := \#\tilde\Delta^*$ such that $E(\tilde\Delta^*) \leq c \eta$, so that $\#\mathrm{I}(\tilde\Delta) \geq n^*$ for $\tilde\Delta$ as in the assertion. Then since $E(\Delta_{N-1})>\eta$,
\[
 \sigma_{n^*} =  E(\tilde\Delta^*) \leq c \eta < c E(\Delta_{N-1}).
\]
However, applying \eqref{eq:treecoarsen}, with \eqref{eq:tccard} we obtain
\[
   E(\Delta_{N-1}) \leq \frac{N}{N-n} \sigma_n \leq c^{-1} \sigma_n 
\]
whenever $n \leq \gamma (N-1)$. Thus, we have $n^* > \gamma (N-1)$ and consequently 
\[
  \#\mathrm{I}(\Delta_{N}) = N < \gamma^{-1} n^* + 1 \leq \gamma^{-1} \#\mathrm{I}(\tilde\Delta) + 1 \leq (\gamma^{-1} + 1) \#\mathrm{I}(\tilde\Delta) . \qedhere
  \]
\end{proof}

From Corollary \ref{cor:treecoarsencard} we can now derive the particular quasi-optimality property required by the adaptive scheme.

\begin{cor}
Let $\omega_0,\omega_1$ with  $0<\omega_0 <  \omega_1 < 1$ be given, let $\Lambda^{\flat}$ be the result of Algorithm \ref{alg:tree} with $\eta = ( 1 -\omega_0^2) \norm{\br}^2$. Then \eqref{bulktree} holds with $\tilde C$ depending only on $\omega_0$ and $\omega_1$.
\end{cor}

\begin{proof}
Let $\Delta_N$ be as computed in Algorithm \ref{alg:tree}. Note that $\norm{\br}^2 = \norm{\br|_{\Lambda^{\flat}}}^2 + \norm{\bP{\Theta} \br - \bP{\mathrm{I}(\Delta_N)}\br }^2$, where
\[
  \norm{\bP{\Theta} \br - \bP{\mathrm{I}(\Delta_N)}\br }^2 = E(\Delta_N) \leq ( 1 -\omega_0^2) \norm{\br}^2.
\]
Let any $\tilde\Lambda \in \FStree$ with $\tilde\Lambda\supset \Lambda^0$ and $\norm{\br|_{\tilde\Lambda} } \geq \omega_1 \norm{\br}$ be given. 
Let 
\[
 \tilde\Delta = \Delta_0 \cup  \{ \delta \colon \delta \in \mathrm{C}( \delta' )\text{ for a $\delta' \in \tilde\Lambda \setminus \Lambda^0$}  \}, 
 \]
which is the proper tree with roots $\Delta_0$ containing $\tilde\Lambda \setminus \Lambda^0$ and all children of its elements, so that $\mathrm{I}(\tilde\Delta) = \tilde\Lambda \setminus \Lambda^0$. We then have
\begin{equation}\label{eq:deltacond}
  \norm{\bP{\Theta} \br - \bP{\mathrm{I}(\tilde\Delta)}\br }^2 = E(\tilde\Delta) \leq ( 1 - \omega_1^2 ) \norm{\br}^2.
\end{equation}
From Corollary \ref{cor:treecoarsencard} with $c = (1 - \omega_1^2) / ( 1 - \omega_0^2)$, for any proper tree $\tilde\Delta$ with roots $\Delta_0$ such that
\eqref{eq:deltacond} holds, we have
\[
  \# (\Lambda^{\flat} \setminus \Lambda^0) =  \# \mathrm{I}(\Delta_N)  \leq \tilde C\# \mathrm{I}(\tilde \Delta) = \tilde C \# ( \tilde \Lambda\setminus \Lambda^0)
  \]
 with $\tilde C$ depending on $K$, $\omega_0$, and $\omega_1$.
 \end{proof}
 
 \begin{remark}\label{rem:binning}
 	In the given form, Algorithm \ref{alg:tree} requires $\mathcal{O}(\#(\Lambda^+\setminus \Lambda^0) \log \#(\Lambda^+\setminus\Lambda^0))$ operations due to the requirement of sorting the values $\tilde e(\delta)$, $\delta \in \Lambda^+\setminus \Lambda^0$. As noted in \cite[Rem.~2.2]{B:18}, the sorting can be replaced by a binary binning, where the $\tilde e(\delta)$ are sorted into bins corresponding to ranges of values of the form $[2^{-p} \max_\delta \tilde e(\delta), 2^{-p-1} \max_\delta \tilde e(\delta))$, $p \in \N_0$. In this case, \eqref{eq:treecoarsen} is replaced by
 	\[
 	   E(\Delta_k) \leq \frac{k+n+1}{k-n+1} \sigma_n, \quad 0 \leq n \leq k \leq N,
 	\] 
 	and the statement of Corollary \ref{cor:treecoarsencard} follows in the same manner with $\gamma = \frac12 (1-c)$. This variant of Algorithm \ref{alg:tree} requires $\mathcal{O}(\#(\Lambda^+ \setminus \Lambda^0))$ operations.
 \end{remark}
 
 \subsection{Galerkin solver}
 For an implementation of $\GALSOLVE$, the simplest option is an iterative scheme with inexact residual approximations by $\RESAPPROX$, where the evaluation in step \textbf{(vi)} is restricted to indices in $\Lambda$.
 
However, a potentially more efficient alternative is provided by the defect correction strategy of \cite{GHS}: starting from a sufficiently accurate approximation of the initial Galerkin residual, an iterative scheme using a fixed approximation of the operator is used to compute a correction. The resulting procedure $\GALSOLVE$ is stated in Algorithm \ref{galsolve}; it relies on the subroutine $\GALAPPLY$ specified in Algorithm \ref{galapply}.

\begin{algorithm}[t]
\caption{~~$\mathbf{\tilde v} = \GALSOLVE(\Lambda, \bv, \delta, \varepsilon)$, where \M{$\treesupp \bv \subseteq \Lambda \in \FStree$, $\#\Lambda < \infty$}, $\varepsilon>0$,and $\delta>0$ such that $\norm{(\bB \bv - \bbf)|_\Lambda} \leq \delta$}
\flushleft Let $\supp \bv \subseteq \Lambda= \{ (\nu,\lambda) \colon \nu \in F,\lambda \in S_\nu \}$, $F \subset \cF$, $S_\nu\in \Stree$ and $v_\nu = \sum_{\lambda \in S_\nu} \mathbf{v}_{\nu,\lambda} \psi_\lambda$
\begin{enumerate}[{\bf(i)}]
\item Determine $(M(\nu))_{\nu \in F}$ by $\APPLY(\bv ; \frac{\varepsilon}{3})$ as in Algorithm \ref{Apply_nosort}, with output restricted to $F$, and set $\br_0 = \GALAPPLY(\Lambda, \bv, (M(\nu))_{\nu \in F}, f)$
\item With the smallest $L \in \N_0$ such that $r_\bB^{-1} C_\bB 2^{- \alpha L} \leq \frac{\varepsilon}{3(\varepsilon +\delta) }$, take $(\tilde M(\nu))_{\nu \in F}$ such that for all $\bw$ with $\suppF \bw \subseteq F$,
\[
    ( \bB_L \bw)_\nu = \sum_{(\mu,\nu') \in \tilde M(\nu) }    (\bM_{\mu})_{\nu,\nu'} \,\bA_{\mu} \, \bw_{\nu'} , \quad \nu \in F\,
\]
\item Use the conjugate gradient method to find $\mathbf{s}$ such that $\norm{  \br_0 + \mathbf{B}_L \mathbf{s} } \leq \frac{\varepsilon}{3}$, 
where for $\bw$ with $\supp \bw \subseteq \Lambda$,
\[
  \mathbf{B}_L \bw =  \GALAPPLY(\Lambda, \bw, (\tilde M(\nu))_{\nu \in F}, 0),
\]
and set $\mathbf{\tilde v} = \bv + \mathbf{s}$
\end{enumerate}
\label{galsolve}
\end{algorithm}

\begin{algorithm}[t]
\caption{~~$\mathbf{w} = \GALAPPLY(\Lambda, \bv, (M(\nu))_{\nu \in F}, f_0)$, where \M{$\treesupp \bv \subseteq \Lambda \in \FStree$ and $\#\Lambda < \infty$}}
\flushleft Let $\supp \bv \subseteq  \Lambda = \{ (\nu,\lambda) \colon \nu \in F,\lambda \in S_\nu \}$, $F \subset \cF$, $S_\nu\in \Stree$ and $v_\nu = \sum_{\lambda \in S_\nu} \mathbf{v}_{\nu,\lambda} \psi_\lambda$, and $f_0$ piecewise polynomial with $\cT(f_0) < \infty$
\begin{enumerate}[{\bf(i)}]
\item For each $\nu \in F$, transform $v_\nu$ to piecewise polynomials on tilings $\cT(v_\nu)$ by applying Algorithm \ref{tilingtransform}

\item For each $\nu \in F$

\noindent
\quad Initialize $\hat w_\nu = -\delta_{0,\nu} f_0$

\noindent
\quad For each $(\mu, \nu') \in M(\nu)$\vspace{6pt}

\noindent\vspace{6pt}
\quad\quad $\displaystyle  \hat w_\nu \gets \hat w_\nu +  (\bM_{\mu})_{\nu,\nu'}  \sum_{T \in \cT_{\neq 0}(A_\mu v_{\nu'})} A_\mu v_{\nu'}\big|_{T}$

\item For each $\nu \in F$,
 use Algorithm \ref{tilingtransform} to transform $\hat w_\nu$ to its representation on the minimal tiling $\cT(\hat w_\nu)$, and set $S_\nu^+ = \cS(\cT(\hat w_\nu), 0)$.
 Determine $\Phi_\nu = \{ \varphi_\lambda \}_{\lambda \in \Sigma_\nu}$ as the corresponding locally single-scale basis with $\linspan \Phi_\nu \supseteq \linspan \{ \psi_\lambda \}_{\lambda \in S^+_\nu}$, $\Sigma_\nu = \Sigma(S_\nu^+)$, according to \eqref{eq:singlescaledef}, evaluate the integrals
\[
  \bs_{\nu,\lambda} = \hat w_{\nu}(\varphi_\lambda) \quad\text{for $\nu \in F$, $\lambda \in \Sigma_\nu$,}
\]
set $\tilde\bw_\nu = \bT_{S^+_\nu}^\top \bs_\nu$ and define $\bw_\nu$ by
\[ 
 \bw_{\nu,\lambda} = \begin{cases} \tilde\bw_{\nu,\lambda}, & \lambda \in S_\nu^+ \cap S_\nu, \\ 0 , & \text{otherwise,} \end{cases} \qquad \text{ for $\lambda \in \cS$}
\]
\end{enumerate}
\label{galapply}
\end{algorithm}

\begin{prop}\label{prop:galerkin}
Let $\mathbf{\tilde v} = \GALSOLVE(\Lambda, \bv, \delta, \varepsilon)$, then $\norm{(\bB \mathbf{\tilde v} - \bbf)|_\Lambda } \leq \varepsilon$, and for any $s>0$ with $s < \frac\alpha{d}$, the required number of arithmetic operations is bounded up to a constant by
\begin{multline}
\label{gsopest}
  \Bigl(   \#\cT(f)  +   \#\Lambda      +  \varepsilon^{-\frac1s} \bignorm{\bigl( \norm{\bv_\nu} \bigr)_{\nu \in \cF}}_{\cA^s}^{\frac1s}    \Bigr) \\ \times 
  \bigl(1+ g(\delta/\varepsilon) +   \abs{\log \varepsilon} + \log \bignorm{\bigl( \norm{\bv_\nu} \bigr)_{\nu \in \cF}}_{\cA^s} 
  +  \max_{\nu \in F^+} \#\supp \nu \bigr)  ,
\end{multline}
where $g\colon \R^+ \to \R^+$ is a nondecreasing function.
\end{prop}

\begin{proof}
The bound on $\norm{(\bB \mathbf{\tilde v} - \bbf)|_\Lambda }$ follows from \cite[Thm.~2.5]{GHS}.
Concerning the costs of step {\bf (i)} of Algorithm \ref{galsolve}, for $( M(\nu))_{\nu \in F}$ we obtain from Proposition \ref{semidiscrapply} the estimate
\begin{equation}
   \sum_{\nu \in F} \# M(\nu) \lesssim  \varepsilon^{-\frac1s}  \bignorm{\bigl( \norm{\bv_\nu} \bigr)_{\nu \in \cF}}_{\cA^s}^{\frac1s} \,.
\end{equation}
Proceeding as in the proof of Theorem \ref{thm:resapprox}, the total number of arithmetic operations for this step is bounded by a fixed multiple of
\begin{equation*}
     \#\cT(f)  +  \Bigl( \#\Lambda      +   \varepsilon^{-\frac1s}  \bignorm{\bigl( \norm{\bv_\nu} \bigr)_{\nu \in \cF}}_{\cA^s}^{\frac1s} \Bigr)  \Bigl( 1 +  \abs{\log \varepsilon} + \log \bignorm{\bigl( \norm{\bv_\nu} \bigr)_{\nu \in \cF}}_{\cA^s} + \max_{\nu \in F^+} \#\supp \nu\Bigr)  \,.
\end{equation*}

We now consider the costs of Algorithm \ref{galapply} with $(\tilde M(\nu))_{\nu \in F}$ as determined in step {\bf (ii)} of Algorithm \ref{galsolve}.
Note that $L$ is a nondecreasing function of $\delta/\varepsilon$.
Moreover,
\begin{equation}
  \sum_{\nu \in F} \# \tilde M(\nu) \lesssim 2^{dL} \# \suppF \bv  \,.
\end{equation}
Again proceeding as in the proof of Theorem \ref{thm:resapprox}, one verifies that the number of arithmetic operations for one application of $\GALAPPLY$ in step {\bf (iii)} is bounded by a multiple of
\[
 \#\cT(f) +  ( 1 +  L ) \sum_{\nu \in F} \sum_{k = 0}^L \bigl( \#\cT(v_\nu) + (k+1) 2^{dk} \bigr)
\]
using the corresponding bounds for $\# \cT(\hat w_\nu)$ with $\hat w_\nu$, $\nu \in F$, as in Algorithm \ref{galapply}.
Accordingly, the costs of one iteration of the solver in step {\bf (iii)} are bounded by a multiple of
$
   \bigl( 1 +  L \bigr) \bigl(  \#\cT(f) + L \# \Lambda + L 2^{dL} \bigr)
$,
and the number of iterations required for the solver depends only on $\delta/\varepsilon$.
\end{proof}

\section{Optimality}

\begin{algorithm}[t]
	\caption{Adaptive Galerkin method}\label{AGM2}
	\flushleft Let $0< \omega_0 < \omega_1 < 1$, $\zeta, \gamma >0$ as in \eqref{eq:AGMparams}, $\bu^0 = 0$, and $\Lambda^0 = \emptyset$; formally set $\norm{\br^{-1} } := r_\bB^{-1} \norm{\bbf}$
\begin{tabbing}
AAA\= AAA \= \kill
 For $k = 0, 1, 2, \ldots$, perform the following steps: \\
\> $( \tilde \Lambda^{k+1}, \br^k, \eta_k, b_k) := \RESAPPROX(\bu^k ; \zeta, \frac{\zeta}{1 + \zeta} \norm{ \br^{k-1} }, \varepsilon )$\\
\> if $b_k \leq \varepsilon$, \\
\>\> return $\bu^k$\\
\> $\Lambda^{k+1} := \TREEAPPROX(\Lambda^k, \tilde\Lambda^{k+1}, \br^k, (1-\omega_0^2) \norm{\br^k}^2)$  \\
\> $\bu^{k+1} := \GALSOLVE(\Lambda^{k+1}, \bu^k, b_k, \gamma \norm{\br^k})$
\end{tabbing}
\end{algorithm}

We now consider the computational complexity of the basic adaptive scheme of Algorithm \ref{AGM} using the residual approximation of Algorithm \ref{optscheme}, tree coarsening by Algorithm \ref{alg:tree} and Galerkin solves by Algorithm \ref{galsolve}, which is summarized in Algorithm \ref{AGM2}. We proceed in two steps. First, we estimate the cardinality of discretizations that are generated in terms of the achieved error tolerance. With the residual approximation and tree coarsening schemes in place, this can be done by techniques from \cite{GHS} and \cite{S09}. In the second step, we consider the computational complexity of the method, where additional specifics of our countably-dimensional setting come into play.

In this section, we frequently use the condition number $\kappa(\bB) = \norm{\bB} \norm{\bB^{-1} }$ with respect to the spectral norm, as well as the energy norm
\[
  \norm{ \bv }_{\bB} = \sqrt{  \langle \bB \bv, \bv \rangle }, \quad \bv \in \ell_2(\cF\times\cS),
\]
associated to the mapping $\bB$ defined in \eqref{seqform}.
To ensure optimality of the scheme, we require the following assumptions on the parameters $\zeta \in (0,\frac12)$, $0 < \omega_0 < \omega_1 < 1$, and $\gamma>0$ of Algorithm \ref{AGM}: 
\begin{equation}\label{eq:AGMparams}
\begin{gathered}
  0< \zeta < \frac{\omega_0}{\omega_0 + 1} , \\
  \omega_1 (1 - \zeta) + \zeta <  (1 - 2\zeta) \kappa(\bB)^{-\frac12}, \\
  0<\gamma < \frac{(1 - \zeta)\omega_0 - \zeta}{(1 + \zeta) \kappa(\bB)}.
\end{gathered}
\end{equation}
Note that the requirements on $\zeta$ ensure that the upper bound for $\gamma$ is positive.

The main result of this work is the following theorem, which combines the above mentioned cardinality and complexity estimates. The proof is given in the following two subsections.

\begin{theorem}\label{complexitythm}
Let $f \in L_2(D)$ be piecewise polynomial with $\#\cT(f)<\infty$, let $\{ \theta_\mu\}_{\mu \in \cM}$ satisfy Assumptions \ref{ass:wavelettheta}, and let Assumptions \ref{ass:pwpolypsi}, \ref{ass:pwpolytheta} hold. 
Let $0 < s < \frac{\alpha}{d}$ and $\norm{\bu}_{\mathrm{t},p}< \infty$ for $p = \left( s + \frac12\right)^{-1}$. Then for each $\varepsilon>0$, Algorithm \ref{AGM2} with parameters satisfying \eqref{eq:AGMparams} outputs an approximation $\bu^k$ for some $k\in\N$ with $\norm{\bu-\bu^k}_{\ell_2}\leq\varepsilon$, such that the following holds:
\begin{enumerate}[{\rm(i)}] 
\item There exists $C>0$ independent of $\varepsilon$ and $\bu$, but depending on $s$, such that 
\[ \M{ \#\treesupp \bu^k } \leq C\, \varepsilon^{-\frac1s} \norm{\bu}_{\mathrm{t},p}^{\frac1s}. \]
\item The scheme can be realized such that with a $C>0$ independent of $\varepsilon$ and $\bu$, the number of operations required to compute $\bu^k$ is bounded by
\[ C\bigl(1 + \varepsilon^{-\frac1s} \norm{\bu}_{\mathrm{t},p}^{\frac1s}(1 + \abs{\log \varepsilon} + \log  \norm{\bu}_{\mathrm{t},p})\bigr) \,.
\]
\end{enumerate}
\end{theorem}

\begin{remark}\label{remark1d}
 If in addition to the assumptions of Theorem \ref{complexitythm}, $d\geq 2$ and $D$ is convex, then Proposition \ref{prop:treeapprox} applies, and thus the statement of Theorem \ref{complexitythm} holds for any $s < \frac{\alpha}{d}$. In particular, for any such $s$, the number of operations required by Algorithm \ref{AGM2} is bounded by $C \varepsilon^{-1/s}$ with a $C>0$ depending on $u$ and $s$. As a consequence of Remark \ref{rem:1dcase}, in the special case $d=1$ the statement holds only for $s < \frac{2}{3} \alpha$.
\end{remark}

\subsection{Cardinality of discretization subsets}

In preparation of the proof of statement (i) in Theorem \ref{complexitythm}, we use ideas from \cite[Lemma~2.1]{GHS}, \cite{S14}, 
and \cite[Prop.\ 4.2]{S09} in order to relate error reduction to cardinality in our present setting of tree approximation.

\begin{lemma}
\label{ghslemma}
	Let $\beta \in (0, \norm{\bB}^{-\frac12})$, $\omega\in (0, \kappa(\bB)^{-\frac12}(1 - \norm{\bB} \beta^2)^{\frac12}]$ and $\bw\in\ell_2(\cF\times\cS)$ such that $\supp \bw \subseteq \Lambda_0 \in  \FStree$. Then the smallest $\Lambda\supseteq \Lambda_0$ with $\Lambda \in  \FStree$ and
	\begin{equation}\label{rescond}
		\norm{(\bB\bw - \bbf)|_\Lambda} \geq \omega \norm{\bB\bw - \bbf} 
	\end{equation}
	satisfies
	\begin{equation}
		\# (\Lambda \setminus \Lambda_0) \leq \min \bigl\{  \#\bar\Lambda\colon \bar\Lambda \in \FStree,\; \min_{\supp \bv \subseteq \bar \Lambda} \norm{\bu - \bv} \leq \beta \norm{\bu - \bw}_{\bB}  \bigr\}.
	\end{equation}
\end{lemma}

\begin{proof}
	With $N:=\min \{  \#\bar\Lambda\colon \bar\Lambda \in \FStree,\; \min_{\supp \bv \subseteq \bar \Lambda} \norm{\bu - \bv} \leq \beta \norm{\bu - \bw}_{\bB} \}$, let $\bu_N$ be a best $N$-term tree approximation with $\supp \bu_N \subseteq \bar\Lambda_N \in \FStree$, $\#\bar\Lambda_N = N$, of $\bu$ such that $\norm{\bu - \bu_N} \leq \beta \norm{\bu - \bw}_{\bB}$. 
	With $\hat\Lambda \coloneqq \Lambda_0 \cup \bar\Lambda_N \in  \FStree$, the Galerkin solution $\bu_{\hat\Lambda}$ satisfies
	\[
	   \norm{\bu - \bu_{\hat\Lambda}}_{\bB} \leq \norm{\bu - \bu_N}_{\bB} \leq \norm{\bB}^{\frac12} \norm{\bu - \bu_N} \leq \norm{\bB}^{\frac12} \beta\norm{\bu - \bw}_{\bB}.
	\]
	By Galerkin orthogonality, $\norm{\bu - \bw}^2_\bB \leq \norm{\bu_{\hat\Lambda} - \bw}_{\bB}^2 + \norm{\bB} \beta^2 \norm{\bu - \bw}^2_\bB$, and therefore
	\[
	    \norm{\bu_{\hat\Lambda}  - \bw}_\bB \geq ( 1 - \norm{\bB} \beta^2)^{\frac12} \norm{\bu - \bw}_\bB.
	\]
	This gives
	\begin{align*}
		\norm{( \bB \bw - \bbf )|_{\hat\Lambda}} & = \norm{( \bB \bw -  \bB \bu_{\hat\Lambda} )|_{\hat\Lambda}} \geq \norm{\bB^{-1}}^{-\frac12} \norm{ \bw -  \bu_{\hat\Lambda} }_{\bB} \\
		 & \geq \norm{\bB^{-1}}^{-\frac12} ( 1 - \norm{\bB} \beta^2)^{\frac12} \norm{\bu - \bw}_\bB \\
		 &\geq \kappa(\bB)^{-\frac12} ( 1 - \norm{\bB} \beta^2)^{\frac12}  \norm{\bB \bw - \bbf } \\
		 &\geq \omega \norm{ \bB \bw - \bbf }.
	\end{align*}
	By definition of $\Lambda$ and since $\hat\Lambda \supseteq \Lambda_0$, we arrive at $
	  \#(\Lambda\setminus\Lambda_0) \leq \#(\hat\Lambda\setminus\Lambda_0) \leq N
	 $.
\end{proof}

\begin{lemma}\label{lmm:card}
 Let the parameters of Algorithm \ref{AGM2} satisfy \eqref{eq:AGMparams}.
  Then for the iterates $\bu^k$ with $\supp \bu^k \subseteq \Lambda^k$ one has
 \[
         \norm{ \bu - \bu^{k+1} }_\bB \leq \rho \norm{ \bu - \bu^k }_\bB
 \]
 with $\rho =  \sqrt{  1 -  {((1-\zeta)\omega_0 - \zeta)^2} \kappa(\bB)^{-1} +  \gamma^2 ( 1+\zeta)^2 \kappa(\bB)   } \in (0,1)$, and
 \[
   \# (\Lambda^{k+1} \setminus \Lambda^k) \lesssim 
    \min \bigl\{  \#\bar\Lambda\colon \bar\Lambda \in \FStree,\; \min_{\supp \bv \subseteq \bar \Lambda} \norm{\bu - \bv} \leq \beta \norm{\bu - \bu^k}_{\bB}  \bigr\}.
 \]
\end{lemma}

\begin{proof}
 By Theorem \ref{thm:resapprox}, the output of $\RESAPPROX$ in Algorithm \ref{AGM2} satisfies 
 \[
   \norm{ \br^k - (\bB \bu^k -\bbf) } \leq \zeta \norm{\bB \bu^k -\bbf}. 
 \]
 As a consequence,
 \[
 \begin{aligned}
   \norm{ (\bB \bu^k - \bbf)|_{\Lambda^{k+1}}} &\geq \norm{ \br^k |_{\Lambda^{k+1}}} - \norm{ \br^k - ( \bB \bu^k - \bbf) }  \\
     & \geq \omega_0 \norm{\br^k} - \norm{ \br^k - ( \bB \bu^k - \bbf) } \\
     & \geq \omega_0   \norm{ \bB \bu^k - \bbf} - (\omega_0 + 1)  \norm{ \br^k - ( \bB \bu^k - \bbf) } \\
     & \geq \bigl( \omega_0 - \zeta (\omega_0 + 1) \bigr) \norm{ \bB \bu^k - \bbf} .
 \end{aligned}
 \]
 By Lemma \ref{lmm:saturation}, for the Galerkin solution $\bu_{\Lambda^{k+1}}$ on $\Lambda^{k+1}$ we thus have
 \[
    \norm{ \bu - \bu_{\Lambda^{k+1}} }_\bB \leq \left( 1 -  \frac{((1-\zeta)\omega_0 - \zeta)^2}{\kappa(\bB)}  \right)^{\frac12} \norm{ \bu - \bu^k}_\bB.
 \]
 Moreover,
 \[
 \begin{aligned}
   \norm{ \bu_{\Lambda^{k+1}} - \bu^{k+1} }_\bB &\leq \norm{\bB^{-1}}^{\frac12} \norm{ (\bbf - \bB\bu^{k+1} )|_{\Lambda^{k+1}} } \leq
     \norm{\bB^{-1}}^{\frac12}\gamma\norm{\br^{k}}    \\
     &  \leq \norm{\bB^{-1}}^{\frac12}\gamma ( 1 + \zeta) \norm{ \bbf - \bB \bu^{k} }
      \leq \gamma (1 + \zeta) \kappa(\bB)^{\frac12} \norm{\bu - \bu^k}_\bB,
   \end{aligned}
 \]
 and by Galerkin orthogonality,
 \[
 \begin{aligned}
   \norm{ \bu - \bu^{k+1} }_\bB^2 &=  \norm{ \bu - \bu_{\Lambda^{k+1}} }_\bB^2 +  \norm{ \bu_{\Lambda^{k+1}} - \bu^{k+1} }_\bB^2   \\
    &\leq \Bigl(   1 -  {((1-\zeta)\omega_0 - \zeta)^2} \kappa(\bB)^{-1} +  \gamma^2 ( 1+\zeta)^2 \kappa(\bB)   \Bigr)   \norm{\bu - \bu^k}_\bB^2.
 \end{aligned}
 \]
 
 Let  $\hat \omega :=\frac{\omega_1(1-\zeta) + \zeta}{1 - 2\zeta}$.
 By the choice of $\omega_1$, there exists $\beta \in (0, \norm{\bB}^{-\frac12})$ such that $\hat\omega \leq  \kappa(\bB)^{-\frac12}(1 - \norm{\bB} \beta^2)^{\frac12}$.
 Let $\hat\Lambda \in  \FStree$ with $\hat\Lambda \supset \Lambda^k$ be of minimal cardinality such that
 \[
   \norm{( \bB \bu^k - \bbf )|_{\hat\Lambda} } \geq \hat\omega \norm{ \bB \bu^k - \bbf  } .
 \]
 Then 
 \[
   \begin{aligned}
     \norm{\br^k|_{\hat\Lambda}} &\geq  \norm{(\bB \bu^k - \bbf)|_{\hat\Lambda}} - \norm{ \br^k - (\bB\bu^k - \bbf) } \\
      &\geq \hat \omega \norm{ \bB\bu^k - \bbf }  - \norm{ \br^k - (\bB\bu^k - \bbf) }  \\
       & \geq \hat\omega \norm{\br^k} - (\hat\omega + 1)  \norm{ \br^k - (\bB\bu^k - \bbf) }  \\
       &\geq \left(  \hat\omega - \frac{(\hat\omega + 1) \zeta}{1 - \zeta} \right) \norm{\br^k}   \\
       & = \omega_1 \norm{\br^k}.
   \end{aligned}
 \]
With Lemma \ref{ghslemma} and \eqref{bulk2}, we thus obtain
 \[
 \begin{aligned}
    \#( \Lambda^{k+1} \setminus \Lambda^k)  &  \lesssim \#(\hat\Lambda\setminus \Lambda^k)    \\
      &\leq 
     \min \bigl\{  \#\bar\Lambda\colon \bar\Lambda \in \FStree,\; \min_{\supp \bv \subseteq \bar \Lambda} \norm{\bu - \bv} \leq \beta \norm{\bu - \bu^k}_{\bB}  \bigr\},
     \end{aligned}
 \]
 completing the proof.
\end{proof}

\begin{proof}[Proof of Theorem \ref{complexitythm}(i)]
From Lemma \ref{lmm:card}, we directly obtain convergence of $\bu^k$ to $\bu$. Moreover, since $\# \Lambda^0 = 0$,
\[
\begin{aligned}
  \#\Lambda^k  & = \sum_{i=1}^k \#(\Lambda^i \setminus \Lambda^{i-1})  \\
  &  \lesssim  \sum_{i=0}^{k-1} \min \bigl\{  \#\bar\Lambda\colon \bar\Lambda \in \FStree,\; \min_{\supp \bv \subseteq \bar \Lambda} \norm{\bu - \bv} \leq \beta \norm{\bu - \bu^{i}}_{\bB}  \bigr\}.
  \end{aligned}
\]
By our assumptions on $\bu$ and by Corollary \ref{cor:treebestapproxnorm}, for any $p>0$ such that $\frac{1}{p} < \frac{\alpha}{d} + \frac12$ and $s = \frac1p - \frac12$,
\[
   \min \bigl\{  \#\bar\Lambda\colon \bar\Lambda \in \FStree,\; \min_{\supp \bv \subseteq \bar \Lambda} \norm{\bu - \bv} \leq \beta \norm{\bu - \bu^{i}}_{\bB}  \bigr\}  \lesssim ( \beta \norm{\bu - \bu^i}_\bB )^{-\frac1s}   \norm{ \bu }_{\mathrm{t},p}^{\frac1s}.
\]
Altogether, using in addition that $\norm{\bu - \bu^{k-1}}_\bB \leq \rho^{k-1-i} \norm{\bu - \bu^i}_\bB$, this gives
\[
  \#\Lambda^k \lesssim   \norm{ \bu - \bu^{k-1}}_\bB^{-\frac1s}  \norm{ \bu }_{\mathrm{t},p}^{\frac1s}  \sum_{i=0}^{k-1} \rho^{\frac1s (k-1 - i)}
   \lesssim  C_s \varepsilon^{-\frac1s}   \norm{ \bu }_{\mathrm{t},p}^{\frac1s} 
\]
with $C_s>0$ independent of $\bu$ and $\varepsilon$, where we have used $\norm{\bu - \bu^{k-1}}_\bB \gtrsim \varepsilon$.
\end{proof}

\subsection{Computational complexity}

For understanding the total number of operations required for the adaptive scheme, in our present setting we need to consider the costs of handling of multi-indices in $\cF$, which can be of arbitrary length, as discussed in Sec.~\ref{sec:multiindices}. \M{Here, with $\Lambda^k$ as in Algorithm \ref{AGM2}, we use the notation}
\[
    \M{ F^k = \bigl\{ \nu \in \cF \colon \text{$(\nu,\lambda) \in \Lambda^k$ for some $\lambda \in \cS$} \bigr\} \,.}
\]
\M{The costs for handling multi-indices} enter into the bounds \eqref{ropest} and \eqref{gsopest} for $\RESAPPROX$ and $\GALSOLVE$, respectively, and thus depend on the largest arising support size of a multi-index. This quantity can be controlled by means of the following simple estimate, by which we can subsequently ensure that the costs for each multi-index operation are of order $\mathcal{O}(1+\abs{\log\varepsilon})$.

\begin{prop}\label{prop:suppest}
\M{For $k\in\N$, at iteration} $k$ of Algorithm \ref{AGM2}, one has $\max_{\nu \in F^k} \#\supp \nu \leq k - 1$.
\end{prop}

\begin{proof}
Starting with $F^1 =\{ 0 \} \subset \cF$, due to the bidiagonal structure of the matrices $\bM_\mu$, we have $\max_{\nu \in F^{k+1}} \#\supp \nu \leq \max_{\nu \in F^k} \#\supp \nu + 1$ for each $k$.
\end{proof}

A comparable and slightly sharper bound on the support of arising multi-indices has also been obtained under different assumptions in \cite[Prop.~2.21]{MR3997838} in the context of sparse interpolation and quadrature for \eqref{affinepde}.

\begin{remark}
	In \cite{Dijkema:09}, related issues concerning indexing costs are addressed for wavelet methods applied to problems of fixed but potentially high dimensionality. There, the costs of the handling of wavelet indices also increase with dimension, but are not coupled to the approximation accuracy by an accuracy-dependent effective dimensionality as in the present case.
	As discussed in \cite[\S6]{Dijkema:09}, for wavelet methods working on unconstrained index sets, additional factors in the computational costs that are logarithmic with respect to the error are also difficult to avoid. For the spatial discretization, this issue is circumvented in our present setting due to the restriction to wavelet index sets with tree structure.
\end{remark}

\begin{proof}[Proof of Theorem \ref{complexitythm}(ii)]
 For the call of $\text{\textsc{ResApprox}}(\bu^k; \zeta, \frac{\zeta}{1+\zeta} \norm{\br^{k-1}}, \varepsilon)$ in \M{iteration} $k$ of Algorithm \ref{AGM2}, let $(\br^k,\eta_k, b_k)$ be the corresponding return values. Let $K$ be the stopping index of Algorithm \ref{AGM2}, that is, $b_K \leq \varepsilon < b_{K-1}$.
 By construction, we have $\eta_k \gtrsim \varepsilon$ for $k = 0,\ldots, K$ and $\eta_k \sim \norm{\br^k}  \sim \norm{\M{\bB} \bu^k -\bbf} \M{\sim \norm{\bu^k - \bu}_\bB}$ for $k=0,\ldots,K-1$. \M{With Lemma \ref{lmm:card}, we obtain $\norm{\bu - \bu^k}_\bB \leq \rho^{k-i} \norm{\bu - \bu^i}_\bB$} and thus $\eta_k \lesssim \rho^{k-i} \eta_i$ for $i< k < K$, which implies $1+\abs{\log \eta_k }\gtrsim k$, and there exists $C>0$ such that $\norm{\bu^k} \leq C$ for all $k$.
 
In step $k$, by Theorem \ref{complexitythm}(i), 
we have 
\begin{equation}\label{eq:suppuk}
  \# \suppF \bu^k \leq \#\supp \bu^k \M{\leq \#\treesupp \bu^k} \lesssim \eta^{-\frac1s}_{k-1}  \norm{\bu}_{\mathrm{t},p}^{\frac1s} \,.
\end{equation}
By Remark \ref{rem:treenormequiv} and \cite[Lemma 4.11]{CDD01},
\[
   \bignorm{\bigl( \norm{\bu_\nu} \bigr)_{\nu \in \cF}}_{\cA^s}  \lesssim  \norm{ \bu^k }_{\cA^s} \lesssim \norm{  \bu }_{\cA^s} +  ( \# \supp \bu^k )^s \norm{ \bu - \bu^k} \lesssim \norm{\bu}_{\mathrm{t},p}\,.
\]

For the number of operations required for evaluating $\br^k$ for each $k$ using Algorithm \ref{optscheme}, we apply Theorem \ref{thm:resapprox}  with $F = F^{k-1}$, $\eta_0 = \frac{\delta}{1+\delta} \norm{\br^{k-1}}$, $\eta_k = \eta$.
Using that $\#\cT(f)\lesssim 1$, and combining Theorem \ref{thm:resapprox} with Proposition \ref{prop:suppest} and  \eqref{eq:suppuk},
 the number of operations required for the evaluation of $\br^k$ can be estimated up to a multiplicative constant by $1 + \eta^{-\frac1s}_{k}   \norm{\bu}_{\mathrm{t},p}^{\frac{1}{s}} (  1 + \abs{ \log \eta_k } + \log \norm{\bu}_{\mathrm{t},p} )$, and the same bound holds for $\#\tilde\Lambda^{k+1}$.

The number of operations for performing $\TREEAPPROX$ on $\br^k$ using binary binning according to Remark \ref{rem:binning} is linear in $\#\tilde\Lambda^{k+1}$.
Concerning the call of $\GALSOLVE$, note first that $\norm{(\bB\bu^k-\bbf)|_{\Lambda^{k+1}}} \leq \norm{\bB\bu^k-\bbf} \leq b_k$.
Since if $b_k>\varepsilon$, we have $b_k \sim \eta_k$, we also obtain $b_k \lesssim \gamma \norm{\br^k}$. Thus the ratio $b_k / (\gamma \norm{\br^k})$ is uniformly bounded, and as a consequence of Proposition \ref{prop:galerkin}, the costs of $\GALSOLVE$ can also be estimated up to a multiplicative constant by $1+\eta^{-\frac1s}_{k}   \norm{\bu}_{\mathrm{t},p}^{\frac{1}{s}} (  1 + \abs{ \log \eta_k } + \log \norm{\bu}_{\mathrm{t},p} )$.
\end{proof}

\section{Numerical Experiments}

The adaptive Galerkin method \M{Algorithm} \ref{AGM2} was implemented for spatial dimensions $d=1,2$ using the Julia programming language, version 1.5.3. The numerical experiments were performed on a single core of a Dell Precision 7820 workstation with Xeon Silver 4110 processor.

For simplicity, we take $\Omega=(0,1)^d$.
For the random fields $a(y)$, we use an expansion in terms of hierarchical hat functions formed by dilations and translations of $\theta(x)=\M{\max\{ 1-|2x-1|, 0 \}}$. Specifically, for $d=1$, $\theta_{\mu}$ with $\mu = (\ell, k)$ is given by 
\begin{equation}\label{theta1d}
\theta_{\ell, k}(x) := c 2^{- \alpha \ell} \theta(2^\ell x - k), \quad k = 0,\ldots,2^\ell-1, \; \ell \in \N_0.
\end{equation}
This yields a wavelet-type multilevel structure \eqref{multilevel1}-\eqref{multilevel2} satisfying Assumptions \ref{ass:wavelettheta}, where
\[
\cM = \{(\ell,k)\colon k = 0,\ldots,2^\ell-1, \; \ell \geq 0 \}
\]
with level parameters $\abs{(\ell,k)} = \ell$.
For $d=2$, we take the isotropic product hierarchical hat functions
\begin{equation}\label{theta2d}
\theta_{\ell, k_1, k_2}(x_1,x_2) := c 2^{- \alpha \ell} \M{\theta(2^\ell x_1 - k_1)\,\theta(2^\ell x_2 - k_2)}, \quad (\ell,k_1,k_2) \in \cM,
\end{equation}
with 
\[
\cM = \left\{ \textstyle(\ell,k_1,k_2) \colon \;\ell \in \N_0,\; k_1, k_2 = 0,\frac{1}{2},\ldots,2^\ell-\frac{3}{2},2^\ell-1  \text{ with } ( k_1\in \N_0 \vee k_2\in \N_0)\right\}.
\]
For the spatial wavelet basis $\Psi$, we use piecewise polynomial $L_2$-orthonormal and continuously differentiable Donovan-Geronimo-Hardin multiwavelets \cite{DGH:99} of approximation order seven.

In the practical implementation of Algorithm \ref{optscheme}, we use some simplifications that have no impact on the observed optimal rates. Specifically, on the one hand, in step \textbf{(vi)} of Algorithm \ref{optscheme}, we directly compute integrals of products of wavelets and piecewise polynomial residuals. In our tests, this is quantitatively favorable, since it avoids some overhead the for multiscale transformations in step \textbf{(vi)}. On the other hand, Galerkin problems are solved by direct application of inexact conjugate gradient iteration in wavelet representation where previously computed matrix entries are cached. 

The quantitative performance of the scheme can also be improved by choosing some of its parameters differently from the values used in the convergence analysis. This is a common observation in such methods (see, e.g., \cite{GHS,Dijkema:09}) relating to the lack of sharpness in various estimates that are used. In particular, $\omega_0$ can be chosen significantly larger than the values allowed by \eqref{eq:AGMparams} without impact on the optimality of the method, but with an improvement of the quantitative performance. Similarly, choosing $C_\bB$ in \eqref{eq:elljdef} larger than a certain value (which is observed to be significantly lower than the one from Proposition \ref{BoundSemiDisc}) does not change the residual estimates, but only increases the computational costs. Moreover, the quantitative performance can also be improved by decreasing the tolerance $\eta$ in step \textbf{(ii)} of Algorithm \ref{optscheme} by a factor different from two. Especially for small $\alpha$, taking this factor as $2^\alpha$ or smaller leads to a more conservative increase in the parameters $\ell_j$ in \eqref{semidiscrwdef}, so that these are not chosen larger than necessary in the final iteration of the loop.

The adaptive scheme is tested with $\alpha = \frac12,\frac23,1,2$ for both $d=1$ and $d=2$. We take $f \equiv 1$ and $c= \frac{1}{10}$ in \eqref{theta1d}, \eqref{theta2d}. The parameters of the scheme are chosen as $\omega_0 = \frac12$, $C_\bB = \frac{1}{100}$, and $\hat\ell = 1$; in step \textbf{(ii)} of Algorithm \ref{optscheme}, we replace $\eta$ by $\eta / 2^{\alpha/d}$.
The results of the numerical tests are shown in Figure \ref{plot:time1d} for $d=1$ and in Figure \ref{plot:time2d} for $d=2$.

The results are compared to the convergence rates that are expected for $\alpha \leq 1$ in view of Theorem \ref{complexitythm} combined with Proposition \ref{prop:treeapprox} for $d=2$ and with Remark \ref{remark1d} for $d=1$.
 For $d=1$, the asymptotic growth of the runtime (in seconds) and the total number of degrees of freedom $\#\Lambda = \#\supp \bu^k$ in terms of the residual error bound $\varepsilon$ is approximately of order $\cO(\varepsilon^{-{3}/{(2\alpha)} })$, which is consistent with the expected limiting rate $\frac23 \alpha$. 
For $d=2$, we instead observe $\cO(\varepsilon^{-{2}/{\alpha} })$, which is consistent with the expected rate $\frac\alpha2$. For both values of $d$, we obtain the analogous result also for $\alpha=2$, which is not covered by the existing approximability analysis.

\begin{figure}
	\includegraphics[width=14.5cm]{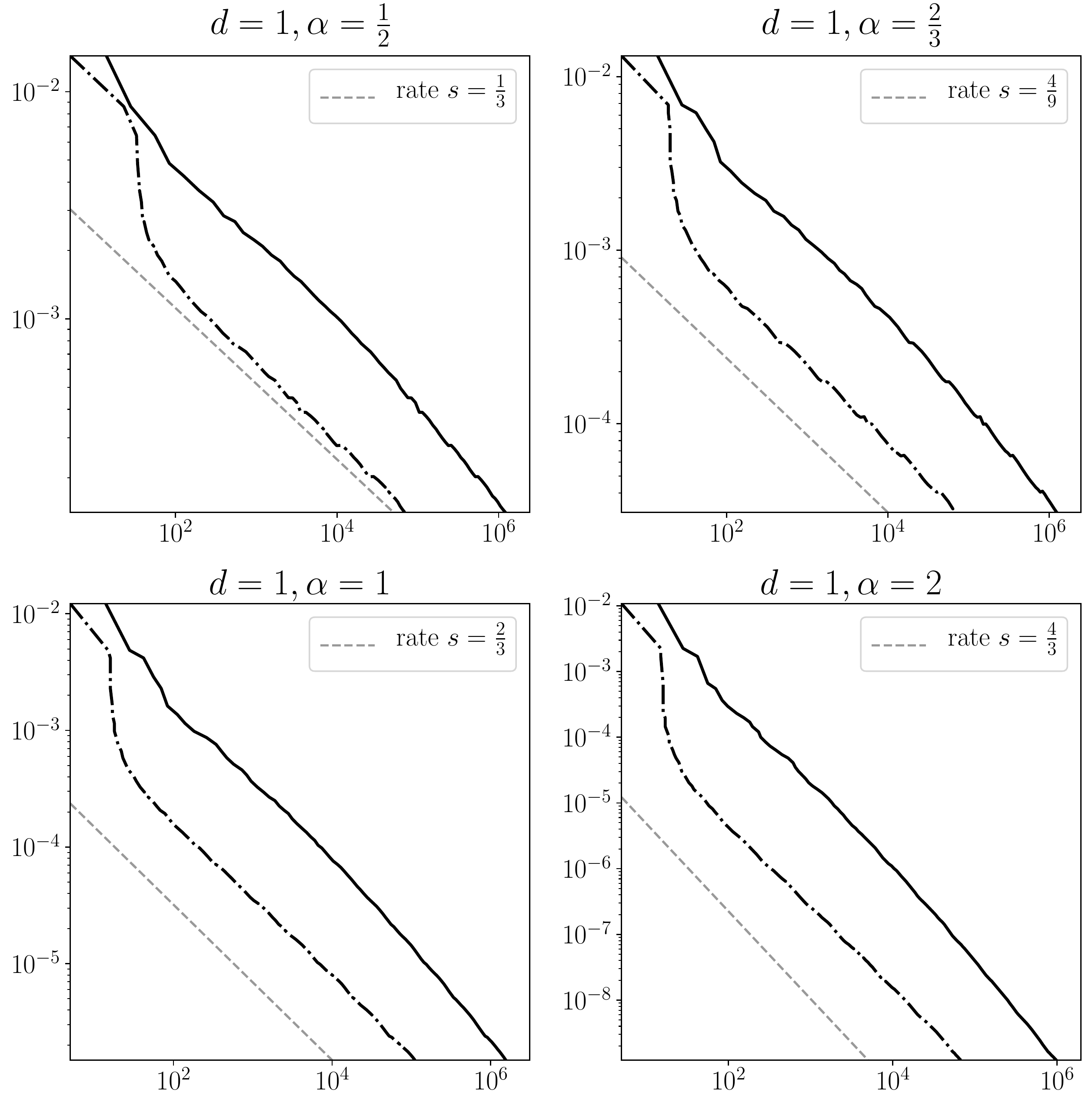}
	\caption{Computed residual bounds for $d=1$ as a function of total number of degrees of freedom of the current approximation of $\bu$ (solid lines) and elapsed computation time (dash-dotted line).}
	\label{plot:time1d}
\end{figure}
\begin{figure}
	\includegraphics[width=14.5cm]{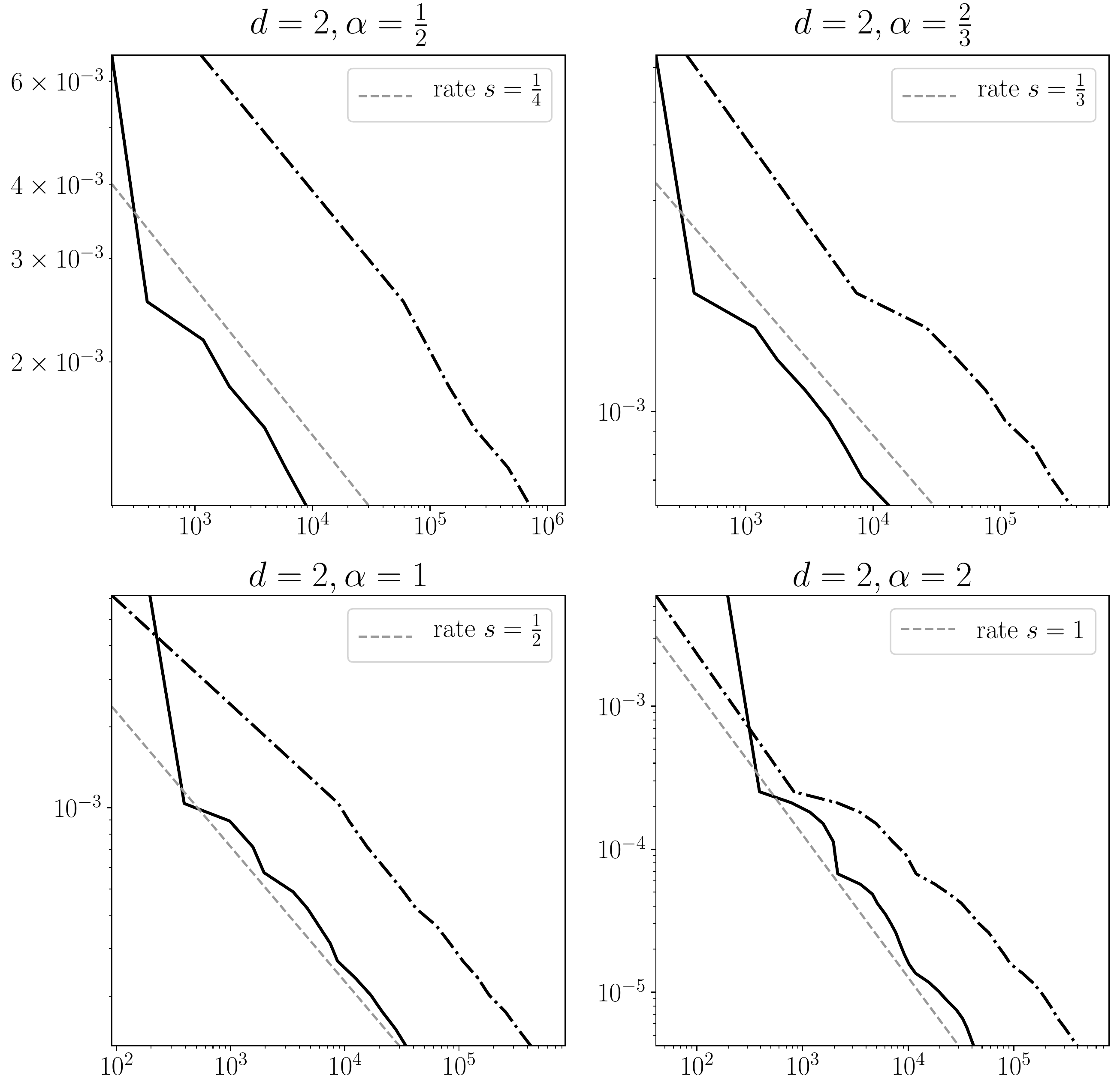}
	\caption{Computed residual bounds for $d=2$ as a function of total number of degrees of freedom of the current approximation of $\bu$ (solid lines) and elapsed computation time (dash-dotted line).}
	\label{plot:time2d}
\end{figure}

\section{Conclusions}

We have shown the adaptive Galerkin method proposed in this work to converge at optimal rates up to $\frac{\alpha}{d}$, where $d$ is the spatial dimension of the diffusion problem \eqref{eq:diffusionequation} and $\alpha$ is the decay parameter in the multilevel expansion of the random diffusion coefficient, which corresponds to the H\"older smoothness of its realizations. The computational costs are guaranteed to scale linearly up to a logarithmic factor with respect to the number of degrees of freedom. To the best of our knowledge, this is the first method with this property in the case where the approximability is limited by the random field rather than by the approximation order of the spatial basis.

Our numerical results confirm the approximability results for $\alpha \in (0,1]$ established in \cite{BCDS}: for $d=2$, we observe a rate $\frac{\alpha}{2}$, whereas in the special case $d=1$, we obtain $\frac{2}{3} \alpha$. The numerical tests also support the conjecture that one has the analogous rates of best approximation for all $\alpha>1$.

On the one hand, the use of a piecewise polynomial wavelet Riesz basis helps to avoid a number of technical issues in the complexity analysis. On the other hand, this also makes the method comparably expensive from a quantitative point of view.
However, it actually generates standard adaptive spline approximations of the Legendre coefficients $u_\nu$ and relies on wavelets mainly for approximating residuals in the appropriate dual norm. The basic construction of the method also carries over to spatial finite element approximations, and a variant based on standard adaptive finite elements will be the subject of a forthcoming work.

\bibliographystyle{amsplain}
\bibliography{BVadaptsg}

\providecommand{\bysame}{\leavevmode\hbox to3em{\hrulefill}\thinspace}
\providecommand{\MR}{\relax\ifhmode\unskip\space\fi MR }
\providecommand{\MRhref}[2]{%
  \href{http://www.ams.org/mathscinet-getitem?mr=#1}{#2}
}
\providecommand{\href}[2]{#2}
\begin{thebibliography}{10}

\bibitem{Ag:15}
M.~S. Agranovich, \emph{Sobolev spaces, their generalizations and elliptic
  problems in smooth and {L}ipschitz domains}, Springer, 2015.

\bibitem{BCDS}
M.~Bachmayr, A.~Cohen, D.~{D\~ ung}, and C.~Schwab, \emph{Fully discrete
  approximation of parametric and stochatic elliptic {PDE}s}, SIAM J. Numer.
  Anal. \textbf{55} (2017), 2151--2186.

\bibitem{BCD}
M.~Bachmayr, A.~Cohen, and W.~Dahmen, \emph{Parametric {PDE}s: {S}parse or
  low-rank approximations?}, IMA Journal of Numerical Analysis \textbf{38}
  (2018), 1661--1708.

\bibitem{BCM}
M.~Bachmayr, A.~Cohen, and G.~Migliorati, \emph{{Sparse polynomial
  approximation of parametric elliptic {PDE}s. {P}art {I}: affine
  coefficients}}, ESAIM Math. Model. Numer. Anal. \textbf{51} (2017), 321--339.

\bibitem{BCM:18}
M.~Bachmayr, A.~Cohen, and G.~Migliorati, \emph{Representations of {G}aussian
  random fields and approximation of elliptic {PDE}s with lognormal
  coefficients}, J. Fourier Anal. Appl. \textbf{24} (2018), 621--649.

\bibitem{MR3177362}
A.~Bespalov, C.~E. Powell, and D.~Silvester, \emph{Energy norm a posteriori
  error estimation for parametric operator equations}, SIAM J. Sci. Comput.
  \textbf{36} (2014), no.~2, A339--A363.

\bibitem{MR4014787}
A.~Bespalov, D.~Praetorius, L.~Rocchi, and M.~Ruggeri, \emph{Convergence of
  adaptive stochastic {G}alerkin {FEM}}, SIAM J. Numer. Anal. \textbf{57}
  (2019), no.~5, 2359--2382.

\bibitem{BPR21}
A.~Bespalov, D.~Praetorius, and M.~Ruggeri, \emph{Convergence and rate
  optimality of adaptive multilevel stochastic {G}alerkin {FEM}}, IMA Journal
  of Numerical Analysis (2021).

\bibitem{MR3519560}
A.~Bespalov and D.~Silvester, \emph{Efficient adaptive stochastic {G}alerkin
  methods for parametric operator equations}, SIAM J. Sci. Comput. \textbf{38}
  (2016), no.~4, A2118--A2140.

\bibitem{MR4114136}
A.~Bespalov and F.~Xu, \emph{A posteriori error estimation and adaptivity in
  stochastic {G}alerkin {FEM} for parametric elliptic {PDE}s: beyond the affine
  case}, Comput. Math. Appl. \textbf{80} (2020), no.~5, 1084--1103.

\bibitem{B:07}
P.~Binev, \emph{Adaptive methods and near-best tree approximation}, Oberwolfach
  Report 29/2007, 2007.

\bibitem{B:18}
\bysame, \emph{Tree approximation for {$hp$}-adaptivity}, SIAM Journal on
  Numerical Analysis \textbf{56} (2018), no.~6, 3346--3357.

\bibitem{BD:04}
P.~Binev and R.~DeVore, \emph{Fast computation in adaptive tree approximation},
  Numerische Mathematik \textbf{97} (2004), no.~2, 193--217.

\bibitem{CDDD:01}
A.~Cohen, W.~Dahmen, I.~Daubechies, and R.~DeVore, \emph{Tree approximation and
  optimal encoding}, Applied and Computational Harmonic Analysis \textbf{11}
  (2001), no.~2, 192--226.

\bibitem{CDD01}
A.~Cohen, W.~Dahmen, and R.~DeVore, \emph{Adaptive wavelet methods for elliptic
  operator equations: {C}onvergence rates}, Mathematics of Computation
  \textbf{70} (2001), no.~233, 27--75.

\bibitem{CDD:03}
\bysame, \emph{Sparse evaluation of compositions of functions using multiscale
  expansions}, SIAM Journal on Mathematical Analysis \textbf{35} (2003), no.~2,
  279--303.

\bibitem{CD}
A.~Cohen and R.~DeVore, \emph{Approximation of high-dimensional parametric
  {PDE}s}, Acta Numerica \textbf{24} (2015), 1--159.

\bibitem{CDS:11}
A.~Cohen, R.~DeVore, and C.~Schwab, \emph{Analytic regularity and polynomial
  approximation of parametric and stochastic elliptic {PDE}'s}, Anal. Appl.
  (Singap.) \textbf{9} (2011), no.~1, 11--47.

\bibitem{MR3952679}
A.~J. Crowder, C.~E. Powell, and A.~Bespalov, \emph{Efficient adaptive
  multilevel stochastic {G}alerkin approximation using implicit a posteriori
  error estimation}, SIAM J. Sci. Comput. \textbf{41} (2019), no.~3,
  A1681--A1705.

\bibitem{Dijkema:09}
T.~J. Dijkema, C.~Schwab, and R.~Stevenson, \emph{An adaptive wavelet method
  for solving high-dimensional elliptic {PDE}s}, Constructive Approximation
  \textbf{30} (2009), no.~3, 423--455.

\bibitem{DGH:99}
G.~C. Donovan, J.~S. Geronimo, and D.~P. Hardin, \emph{Orthogonal polynomials
  and the construction of piecewise polynomial smooth wavelets}, SIAM J. Math.
  Anal. \textbf{30} (1999), 1029--1056.

\bibitem{MR3154028}
E.~Eigel, C.~J. Gittelson, C.~Schwab, and E.~Zander, \emph{Adaptive stochastic
  {G}alerkin {FEM}}, Comput. Methods Appl. Mech. Engrg. \textbf{270} (2014),
  247--269.

\bibitem{MR3423228}
\bysame, \emph{A convergent adaptive stochastic {G}alerkin finite element
  method with quasi-optimal spatial meshes}, ESAIM Math. Model. Numer. Anal.
  \textbf{49} (2015), no.~5, 1367--1398.

\bibitem{GHS}
T.~Gantumur, H.~Harbrecht, and R.~Stevenson, \emph{An optimal adaptive wavelet
  method without coarsening of the iterands}, Mathematics of Computation
  \textbf{76} (2007), no.~258, 615--629.

\bibitem{GittelsonThesis}
C.~J. Gittelson, \emph{Adaptive {G}alerkin methods for parametric and
  stochastic operator equations}, Ph.D. thesis, ETH Z{\"u}rich, 2011.

\bibitem{Gittelson:12}
\bysame, \emph{Representation of {G}aussian fields in series with independent
  coefficients}, IMA Journal of Numerical Analysis \textbf{32} (2012), no.~1,
  294--319.

\bibitem{Gittelson:13}
\bysame, \emph{An adaptive stochastic {G}alerkin method for random elliptic
  operators}, Mathematics of Computation \textbf{82} (2013), 1515--1541.

\bibitem{MR3164144}
\bysame, \emph{Adaptive wavelet methods for elliptic partial differential
  equations with random operators}, Numer. Math. \textbf{126} (2014), no.~3,
  471--513.

\bibitem{RS}
N.~Rekatsinas and R.~Stevenson, \emph{An optimal adaptive wavelet method for
  first order system least squares}, Numerische Mathematik \textbf{140} (2018),
  no.~1, 191--237.

\bibitem{R:99}
V.~S. Rychkov, \emph{On restrictions and extensions of the {B}esov and
  {T}riebel-{L}izorkin spaces with respect to {L}ipschitz domains}, Journal of
  the London Mathematical Society \textbf{60} (1999), no.~1, 237--257.

\bibitem{Schwab:11}
C.~Schwab and C.~J. Gittelson, \emph{Sparse tensor discretization of
  high-dimensional parametric and stochastic {PDE}s}, Acta Numerica \textbf{20}
  (2011).

\bibitem{S04}
R.~Stevenson, \emph{On the compressibility of operators in wavelet
  coordinates}, SIAM Journal on Mathematical Analysis \textbf{35} (2004),
  no.~5, 1110--1132.

\bibitem{S09}
\bysame, \emph{Adaptive wavelet methods for solving operator equations: an
  overview}, Multiscale, nonlinear and adaptive approximation, Springer, 2009,
  pp.~543--597.

\bibitem{S14}
\bysame, \emph{Adaptive wavelet methods for linear and nonlinear least-squares
  problems}, Foundations of Computational Mathematics \textbf{14} (2014),
  no.~2, 237--283.

\bibitem{MR3997838}
J.~Zech, D.~D\~{u}ng, and C.~Schwab, \emph{Multilevel approximation of
  parametric and stochastic {PDE}s}, Math. Models Methods Appl. Sci.
  \textbf{29} (2019), no.~9, 1753--1817.

\end{thebibliography}

\begin{appendix}

\section{Tree Approximation}\label{app:trees}

\begin{proof}[Proof of Lemma \ref{lmm:treeupper}]
Let $\delta$ be any leaf. Let $\delta_0, \dots, \delta_{\ell-1}, \delta_\ell = \delta$ be the ancestors of $\delta$, in order, with $\delta_0$ the only root that is an ancestor of $\delta$. By definition of $\tilde e$,
\[
\tilde e (\delta)^{-1} = e(\delta_{\ell})^{-1} + \tilde e (\delta_{\ell-1})^{-1} 
= e(\delta_{\ell})^{-1} + e (\delta_{\ell-1})^{-1} + \tilde e (\delta_{\ell-2})^{-1}
= \cdots = \sum_{j = 0}^\ell  e(\delta_{j})^{-1}.
\]
Using that $\tilde e(\delta)\le \eta $ and multiplying by $e (\delta) \tilde e (\delta) $, we obtain 
\begin{equation}
\label{treeupper:leafbound}
e (\delta) 
=
\tilde e (\delta) 
\sum_{j = 0}^\ell e (\delta) 
  e(\delta_{j})^{-1}
  \le
\eta
\sum_{j = 0}^\ell \frac{ e (\delta) }{
e(\delta_{j}) }  .
\end{equation}
To take the sum over all leaves $\delta\in \mathrm{L}(\Lambda)$, we consider the subtree 
$
\Delta_{\tilde \delta} = \{ \delta \in \Lambda \colon  \delta \preceq \tilde\delta \}
$
that is rooted at $\tilde\delta \in \Lambda$.
For any leaf $\delta \in \mathrm{L}(\Delta_{\tilde \delta} )\subset  \mathrm{L}(\Lambda ) $ of such a subtree, we consider the contribution 
$\frac{ e (\delta) }{e(\tilde\delta) } $ of the ancestor $\delta_j = \tilde \delta$ to the sum on the \M{right-hand side} of \eqref{treeupper:leafbound}. Thus
\[
\sum_{\delta \in \mathrm{L}(\Lambda) } e (\delta) 
\le
\eta
\sum_{\tilde \delta \in \Lambda } \sum_{\delta \in \mathrm{L}(\Delta_{\tilde \delta}) }
 \frac{ e (\delta) }{
	e(\tilde \delta) }  .
\]
Due to the subadditivity  of $e$, we have  $\sum_{\delta \in \mathrm{L}(\Delta_{\tilde \delta}) } \frac{ e (\delta) }{e(\tilde \delta) } \le 1$ (trivially for $\tilde \delta \in \mathrm{L}(\Lambda)$ and by applying \eqref{subadd} inductively, otherwise), and consequently
\[
\sum_{\delta \in \mathrm{L}(\Lambda) } e (\delta) 
\le
\eta
\sum_{\tilde \delta \in \Lambda } 1 = \eta \#\Lambda. \qedhere
\] 
\end{proof}

\begin{proof}[Proof of Lemma \ref{lmm:treelower}]
We first consider the case that $\delta_0$ is not a root and has a parent $ \delta_0^*$. 
We prove the slightly stronger statement 
\[
e(\delta_0) \ge \eta \left(\#\Lambda_{\delta_0} + \frac{e(\delta_0)}{\tilde e(\delta_0^*)} \right)
\]
by induction on the size of $\Lambda_{\delta_0}$. 

If $\#\Lambda_{\delta_0}=1$, then we have only the node $\delta_0$ in this tree. By definition of $\tilde e$, we have
\[
e(\delta_0) 
= e(\delta_0) \tilde e(\delta_0) \tilde e(\delta_0)^{-1}
= e(\delta_0) \tilde e(\delta_0) (e(\delta_0)^{-1} + \tilde e(\delta_0^*)^{-1})
.
\]
It follows that 
\[
\sum_{\delta \in \mathrm{L}(\Lambda)} e(\delta) = e(\delta_0) 
= \tilde e(\delta_0) \left( 1 + \frac{e(\delta_0)}{\tilde e(\delta_0^*)} \right)
\ge \eta
\left( 1 + \frac{e(\delta_0)}{\tilde e(\delta_0^*)} \right)
.
\]
Now let $\delta_0\in \Lambda$ be any node with a parent $\delta_0^*$, and let $\Lambda_{\delta_0}$ be a subtree of $\Lambda$ rooted at $\delta_0$ such that $\tilde e(\delta) \geq \eta$ for all $\delta \in \Lambda_{\delta_0}$,
and assume that 
\[
e(\tilde \delta_0) \ge \eta \left(\# \Lambda_{\tilde \delta_0} + \frac{e(\tilde\delta_0)}{\tilde e(\tilde \delta_0^*)} \right)
\]
for any node $\tilde \delta_0\in\Lambda$ with a parent $\tilde \delta_0^*$ and any subtree  $\Lambda_{\tilde \delta_0}$ rooted at $\tilde \delta_0$ such that $\tilde e(\delta) \geq \eta$ for all $\delta \in \Lambda_{\tilde \delta_0}$ and $\# \Lambda_{\tilde \delta_0} < \# \Lambda_{\delta_0}$.

Consider any child $\delta \in \mathrm{C}(\delta_0)$, then for the subtree of $\Lambda_{\delta_0}$ rooted at $\delta$, which we will denote by $\Lambda_{\delta}$, we have $\#\Lambda_{\delta} < \# \Lambda_{\delta_0}$ and
$\tilde e(\delta) \geq \eta$ for all $\delta \in \Lambda_{\tilde \delta_0}$. By applying the induction hypothesis to each child of $\delta_0$, we get 
\[
\sum_{\delta\in \mathrm{C}(\delta_0)} e(\delta) \ge\left (\sum_{\delta\in \mathrm{C}(\delta_0)} \# \Lambda_{ \delta} + \frac{\sum_{\delta\in \mathrm{C}(\delta_0)} e(\delta)}{\tilde e( \delta_0)} \right) \eta
\ge\left (\sum_{\delta\in \mathrm{C}(\delta_0)} \# \Lambda_{ \delta} + \frac{e(\delta_0)}{\tilde e( \delta_0)} \right) \eta
.
\]
Using the definition of $\tilde e$, we get
\[
\sum_{\delta\in \mathrm{C}(\delta_0)} e(\delta) \ge
\left (\sum_{\delta\in \mathrm{C}(\delta_0)} \# \Lambda_{ \delta} + 1 + \frac{e(\delta_0)}{ e( \delta_0^*)} \right) \eta
=
\left( \# \Lambda_{ \delta_0}  + \frac{e(\delta_0)}{ e( \delta_0^*)} \right) \eta
.
\]
This concludes the proof in the case that $\delta_0$ has a parent. 

It remains to prove the original statement in the case that $\delta_0$ is a root. If $\# \Lambda_{ \delta_0} = 1$, the statement is trivial. If $\delta_0$ has children in $\Lambda_{ \delta_0}$, then we know from the first part of the proof that
\[
\begin{aligned}
\sum_{\delta\in \mathrm{C}(\delta_0)} e(\delta) &\ge \left (\sum_{\delta\in \mathrm{C}(\delta_0)} \# \Lambda_{ \delta} + \frac{\sum_{\delta\in \mathrm{C}(\delta_0)} e(\delta)}{\tilde e( \delta_0)} \right) \eta  \\
 &\ge \left (\sum_{\delta\in \mathrm{C}(\delta_0)} \# \Lambda_{ \delta} + \frac{e(\delta_0)}{\tilde e( \delta_0)} \right) \eta
= (\# \Lambda_{ \delta_0} ) \eta \,. \qedhere
\end{aligned}
\]

\end{proof}

\end{appendix}

\end{document}